\documentclass[12pt]{amsart}
\usepackage[all]{xy}
\usepackage{hyperref}
\hypersetup{colorlinks=true,allcolors=blue}
\usepackage{amssymb}
\usepackage{amsthm}
\usepackage{hyperref}
\hypersetup{colorlinks=true,linkcolor=blue,citecolor=magenta}
\usepackage{amsmath}
\usepackage{amscd,enumitem}
\usepackage{verbatim}
\usepackage{eurosym}
\usepackage{float}
\usepackage{graphicx}
\usepackage[section]{placeins}
\usepackage{color}
\usepackage{dcolumn}
\usepackage[mathscr]{eucal}
\usepackage[all]{xy}
\usepackage{hyperref}
\usepackage{bbm}
\usepackage[textheight=8.65in, textwidth=6.71in]{geometry}
\usepackage{multirow}
\usepackage{caption}
\newtheorem*{thm*}{Theorem}
\newtheorem*{conj*}{Conjecture}

\newtheorem*{remark}{Remark}
\newtheorem*{ThreeRemarks}{Three Remarks}

\newtheorem*{FiveRemarks}{Five Remarks}
\newtheorem{thm}{Theorem}[section]
\newtheorem{lem}{Lemma}[section]

\newtheorem{prop}[thm]{Proposition}

\newtheorem*{example}{Example}

\newcommand{\ord}{\mathrm{ord}}
\newcommand{\im}{\rm {Im}}
\newcommand{\Z}{\mathbb{Z}}
\newcommand{\Q}{\mathbb{Q}}

\newcommand{\F}{\mathbb{F}}

\newcommand{\leg}[2]{\genfrac{(}{)}{}{}{#1}{#2}}

\makeatletter
\AtBeginDocument{%
  \expandafter\renewcommand\expandafter\subsection\expandafter{%
    \expandafter\@fb@secFB\subsection
  }%
}
\makeatother

\numberwithin{equation}{section}

\begin{document}
\title[Variants of Lehmer's speculation for newforms]{Variants of Lehmer's Speculation for newforms}
\author{Jennifer S. Balakrishnan, William Craig,  Ken Ono, and Wei-Lun Tsai}
\address{Department of Mathematics and Statistics, Boston University,
Boston, MA 02215}
\email{jbala@bu.edu}
\address{Department of Mathematics, University of Virginia, Charlottesville, VA 22904}
\email{wlc3vf@virginia.edu}
\email{ken.ono691@virginia.edu}
\email{tsaiwlun@gmail.com}

\thanks{The first author acknowledges the support of the NSF (DMS-1702196), the Clare Boothe Luce Professorship
(Henry Luce Foundation), a Simons Foundation grant (Grant \#550023), and a Sloan
Research Fellowship. The third author thanks the support of the Thomas Jefferson Fund and the NSF
(DMS-1601306 and DMS-2055118).}

\begin{abstract} In the spirit of 
Lehmer's unresolved speculation on the nonvanishing of Ramanujan's tau-function,
it is natural to ask whether a fixed integer $\alpha$ is a value of $\tau(n)$ or is a
Fourier coefficient $a_f(n)$ of any given newform $f(z)$. We offer a method, which applies to newforms with integer coefficients and trivial residual mod 2 Galois representation, that answers this question for odd $\alpha$.
We determine infinitely many spaces for which the ordinary primes $3\leq \ell\leq 37$  are not  absolute values of coefficients of newforms with integer coefficients, and we obtain many explicit examples for $\tau(n)$.
We also obtain sharp lower bounds for the number of prime factors of such newform coefficients. 
 In the weight aspect, for powers of odd ordinary primes $\ell$, we prove that 
$\pm \ell^m$ is not a coefficient
of any such newform  $f$ with weight $2k>M^{\pm}(\ell,m)$ and even level coprime to $\ell,$ where $M^{\pm}(\ell,m)$ are effectively computable constants that are $O_{\ell}(m).$
\end{abstract}
\maketitle
\section{Introduction and statement of results}

In a paper innocently entitled ``On certain arithmetical functions,'' Ramanujan introduced his tau-function, whose values are the coefficients of the
weight 12 modular form (note: $q:=e^{2\pi i z}$ where $\im(z)>0$)
\begin{equation}
\Delta(z)=\sum_{n=1}^{\infty}\tau(n)q^n:=q\prod_{n=1}^{\infty}(1-q^n)^{24}=q-24q^2+252q^3-1472q^4+4830q^5-\cdots.
\end{equation}
These coefficients have served as a prototype and testing ground for important phenomena in the theory of modular forms.
Their multiplicative properties offered hints of the theory of Hecke operators. Ramanujan's conjectured bounds on their size are famous corollaries of
Deligne's proof of the Weil Conjectures. Furthermore, Ramanujan offered congruences \cite{RamanujanUnpublished, Ramanujan, SerreRamanujan}, such as
\begin{equation}\label{Cong691}
\tau(n)\equiv \sum_{1\leq d\mid n}d^{11}\pmod{691},
\end{equation}
that Serre  \cite{SerreRamanujan} later viewed as glimpses of the theory of modular $\ell$-adic Galois representations.

Despite these important roles, some of the function's most basic properties remain unknown. For example, Lehmer's speculation\footnote{``Lehmer's Conjecture'' is the assertion that $\tau(n)$ never vanishes. To our knowledge, he never formulated such a conjecture, and so we refer to his question as his speculation.}  that $\tau(n)$ never vanishes remains open. Lehmer proved \cite{Lehmer} that if $\tau(n)$ ever vanishes, then there is a prime $p$ for which $\tau(p)=0.$  Using the Chebotarev Density Theorem, Serre \cite{SerreCheb}  established a quantitative result that implies that the set of such primes $p$ (if any)
 has density zero within the primes. Serre's estimate, which holds for weight $\geq 2$ newforms without complex multiplication,  has been improved several times, and  thanks to work by Thorner and Zaman \cite{ThornerZaman} it is now known that
 $$
 \# \{ p\leq	X \ {\text {\rm prime}}\ : \ \tau(p)=0\} \ll \pi(X)\cdot \frac{(\log \log X)^2}{\log X},
 $$
 where $\pi(X)$ is the usual prime counting function.
 Recent work by Calegari and Sardari \cite{CalegariSardari} considers a different aspect; they establish that
 at most finitely many non-CM newforms with fixed tame $p$ level $N$ have vanishing $p$th Fourier coefficient.
 
 We consider a variation of Lehmer's original speculation that has also been the focus of study. For an  odd integer $\alpha$,
 Murty, Murty, and Shorey \cite{MMS} proved (see \cite{MM} for a generalization) that
 $\tau(n)=\alpha$ for at most finitely many $n$. 
 Due to the enormous bounds that  arise in the theory of linear forms in logarithms (the crux of their method), the classification
 of such $n$ has not been carried out for any $\alpha\neq \pm 1$. For $\alpha=\pm \ell$, where $\ell$ is almost any odd prime,
 it is widely believed that there are  no solutions.  However, there are counterexamples, such as
Lehmer's prime value example \cite{LehmerMonthly}
\begin{equation}\label{LehmerPrime}
\tau(251^2)=-80561663527802406257321747.
\end{equation}
\noindent 
Lygeros and Rozier \cite{LR} have subsequently discovered further prime values.

We investigate these questions for even weight newforms with integer coefficients and trivial mod 2 residual Galois representation (i.e. even Hecke eigenvalues for $T(p)$ for primes $p\nmid 2N$, where $N$ is the level).
We obtain a general theorem (see Theorem~\ref{LehmerVariantGeneral}) that theoretically locates those coefficients that are odd prime powers in absolute value for such newforms. For $\tau(n)$, this theorem gives the following criterion, which restricts arguments to explicit finite sets.

\begin{thm}\label{LehmerVariation}
Suppose that $\ell$ is an odd prime for which $\ell \nmid \tau(\ell).$ If $\tau(n)=\pm\ell^m$, with $m\in \Z^{+},$ then $n=p^{d-1},$ where $p$ and $d\mid \ell( \ell^2-1)$
are odd primes. Furthermore,  $\tau(n)=\pm \ell^m$ for at most finitely many $n$.
\end{thm}

Theorem~\ref{LehmerVariation} offers a method for determining whether $|\tau(n)|=\ell^m$ has any solutions, which reduces the
problem to the determination of certain integer points on finitely many algebraic curves.  For $\ell \in \{3, 5, 7\},$ 
examples of these curves include
\begin{equation}\label{HyperEquations}
 Y^2-X^{11}=\pm 3^m,\ \ \ \ \ 
 Y^2-5X^{22}=\pm 4\cdot 5^m \ \ \ \ {\text {\rm and}}\ \ \ \ Y^{3}-5XY^2+6X^2Y-X^3=\pm 7^m.
\end{equation}
By classifying such points when $m=1$, we obtain 
the following theorem.\footnote{The {\it Journal of Number Theory} published the proceedings of the conference ``Modular forms and Drinfeld Modules'' held in 2018 in Pisa, Italy. Paper \cite{PaperJNT}  is an exposition of the third author's  lecture at the conference, and pertains to some of the cases  of Theorem~\ref{Lehmer135} (1). All of the other results in the present paper have not appeared elsewhere. This article is the main reference for the authors' work on variants of Lehmer's speculation.}

\begin{thm}\label{Lehmer135}
For every $n>1$, the following are true. 

\noindent
(1) We have that
$$\tau(n)\not \in \{\pm 1, \pm 3, \pm 5, \pm 7, \pm 13, \pm 17, -19, \pm 23,  \pm 37, \pm 691\}.$$

\noindent
(2)  Assuming the Generalized Riemann Hypothesis, we have that
 $$\tau(n)\not \in 
\left \{ \pm  \ell\ : \ 41\leq  \ell\leq 97  \  {\text {\rm with}}\ \leg{\ell}{5}=-1\right\} \cup
\left \{ -11, -29, -31, -41, -59, -61, -71, -79, -89\right\}.
 $$
\end{thm}

\begin{remark}
This paper\footnote{This paper was first posted to the arXiv on May 20, 2020.}  has stimulated a number of recent works
on variants of Lehmer's speculation.  Many authors have made use of
its contents and strategy to obtain further results extending and generalizing Theorem~\ref{Lehmer135}.
To be precise, Amir and Hatziiliou \cite{AmirHatziiliou}, Amir and Hong \cite{AmirHong}, 
Bennett, Gherga, Patel and Siksek \cite{BGPS}, Dembner and Jain
\cite{DembnerJain}, Hanada and Madhukara \cite{HanadaMadhukara}, and the authors \cite{PaperJNT, PaperLaMatematica} have made use of
Theorem~\ref{LehmerVariation} to obtain explicit extensions and further generalizations of Theorem~\ref{Lehmer135}.
Most notably, Bennett, Gherga, Patel and Siksek (see Theorem 6 of \cite{BGPS}) proved the striking fact  that
$|\tau(n)|\neq \ell^m$ for every prime $3\leq \ell<100$ and every positive integer $m$.
\end{remark}

\noindent
There are infinite families of newforms with even level for which these methods apply for ordinary primes $\ell$ (i.e. $\ell \nmid a_f(\ell)$). The next theorem
offers unconditional results for $3\leq \ell \leq 37,$ when $2k \in \{4, 6, 8, 10\}$ or
 $\gcd(3 \cdot 5\cdot 7, 2k-1)\neq 1$. It also
gives further  results conditional on the Generalized Riemann Hypothesis (GRH).

\begin{thm}\label{LehmerGeneral} 
If  $f(z)=q+\sum_{n=2}^{\infty}a_f(n)q^n\in S_{2k}(\Gamma_0(2N))\cap \Z[[q]]$ is an even weight $2k\geq 4$ newform with trivial mod 2 residual Galois representation, then the following are true for ordinary primes $\ell.$
\begin{enumerate}
\item For every $n>1$ we have $a_f(n)\not \in \{\pm 1\}.$
\item If $2k=4$, then for every $n$ we have
$$
a_f(n)\not \in \left\{\pm \ell \ : \ 3\leq \ell \leq 37\ {\text {\rm prime}} \right\}\setminus\left 
\{ \pm11, -13,17,\pm19,-23,37\right\}.
$$
Assuming GRH, for every $n$ we have
$$
a_f(n) \not \in \{\pm \ell \ : \ 41 \leq \ell \leq 97 \ {\text {\rm prime}}\}\setminus\{-41,-53,-61,-67,\pm71,73,-89\}.
$$
\item If $2k=6$, then for every $n$ we have
$$
a_f(n)\not \in \left\{\pm \ell \ : \ 3\leq \ell \leq 37\  {\text {\rm prime}}\right\}\setminus\left \{ 11,13\right\}.
$$
Assuming GRH, for every $n$ we have
$$
a_f(n) \not \in \{\pm \ell \ : \ 41 \leq \ell \leq 97 \ {\text {\rm prime}}\}\setminus\{-47\}.
$$
\item If $2k=8$, then for every $n$ we have
$$
a_f(n)\not \in \left\{\pm \ell \ : \ 3\leq \ell \leq 37\  {\text {\rm prime}}\right\}.
$$
Assuming GRH, for every $n$ we have
$$
a_f(n) \not \in \{\pm \ell \ : \ 41\leq \ell \leq 97 \ {\text {\rm prime}}\}\setminus\{-71\}.
$$
\item If $2k=10,$ then for every $n$ we have
$$
a_f(n)\not \in \left\{\pm \ell \ : \ 3\leq \ell \leq 37\ {\text {\rm prime}} \right\}.
$$
Assuming GRH, for every $n$ we have
$$
a_f(n) \not \in \{\pm \ell \ : \ 41\leq \ell \leq 97\ {\text {\rm prime}}\}\setminus\{-83\}.
$$
\item  If $\gcd(3 \cdot 5\cdot 7\cdot 11\cdot 13, 2k-1)\neq 1$ and $2k\geq 12$, then for every $n$ we have 
$$
a_f(n)\not \in  \left \{ \pm \ell \ : \ 3\leq \ell <37 \ {\text {\rm prime with}}\ \leg{\ell}{5}=-1\right\}
\cup \{-37\}.
 $$
Moreover, if $2k\neq 16,$ then $a_f(n)\neq 37.$
Assuming GRH, for every $n$ we have
$$
a_f(n)\not \in  \left \{ \pm \ell \ : \ 41\leq \ell \leq 97 \ {\text {\rm prime with}}\ \leg{\ell}{5}=-1\right\}.
 $$
\item If $\gcd(3\cdot 5, 2k-1)\neq 1$ and $2k\geq 12$, then for every $n$ we have
$$a_f(n) \not \in \left \{\pm \ell \ : \  11\leq \ell \leq 31 \  \text{ {\rm prime with }} \leg{\ell}{5}=1\right\}.
$$
Assuming GRH, the range of this set can be expanded to include $\ell \leq 89.$
\item If $7\mid (2k-1)$ and $2k\geq 12$, then for every $n$ we have
$$a_f(n) \not \in \left \{\pm \ell \ : \  11\leq \ell \leq 31 \  \text{ {\rm prime with }} \leg{\ell}{5}=1\right\}.
$$
Assuming GRH, for every $n$ we have
$$
a_f(n)\not \in \{\pm 41, \pm 59, \pm 61, -71, \pm 79, \pm 89\}.
$$
\item If $11\mid (2k-1),$ then for every $n$ we have $a_f(n)\neq -19$, and assuming GRH
we have
 \begin{displaymath}
a_f(n)\not \in 
\left \{ -11, -29, -31, -41, -59, -61, -71, -79, -89\right\}.
\end{displaymath}

\item If $13\mid (2k-1),$ then for every $n$ we have $a_f(n)\neq -11$, and assuming GRH we have
 \begin{displaymath}
a_f(n)\not \in \left \{-19, -29, -31, -41, -59, -61, -71, -79\right\}.
\end{displaymath}
\end{enumerate}
\end{thm}

\begin{FiveRemarks} \ \ \ \newline 
\noindent
(i) Theorem~\ref{LehmerGeneral} applies to all newforms \cite{OnoTaguchi} with integer coefficients with level $2^aN$, where $a\geq 0$ and
$N\in \{1, 3, 5, 15, 17\}$. Moreover,
the result holds  for all odd levels when $a_f(2)$ is even. 

\smallskip
\noindent
(ii) These results follow from Theorem~\ref{LehmerVariantGeneral}, which  constrains
coefficients that are odd prime powers in absolute value. This method extends to arbitrary odd integers by Hecke multiplicativity, thereby giving an algorithm for determining whether a given odd integer is a newform coefficient. 

\smallskip
\noindent
(iii) The proof of Theorem~\ref{LehmerGeneral} (2-6)  locates values $\pm \ell$ that are possible coefficients.
For example, Theorem~\ref{LehmerGeneral} (2) allows weight 4 coefficients to be in the set
$\{\pm 11, -13, 17, \pm 19, -23, 37\}.$
 The proof shows that these values can only occur as one of the following coefficients:
\begin{displaymath}
\begin{split}
&a_f(3^2)=37,\  \  a_f(3^2)=-11,\  \ a_f(3^2)=-23, \ \  a_f(3^4)=19,\ \ a_f(5^2)=19, \\
&a_f(7^2)=-19,\ \ a_f(7^4)=11,\ \   a_f(17^2)=-13, \ \ a_f(43^2)=17.
\end{split}
\end{displaymath}
Similarly, Theorem~\ref{LehmerGeneral} (6) allows a coefficient  of 37 for weight $16,$
which must be $a_f(3^2)=37.$ 

\smallskip
\noindent
(iv) The assumption that $2k\geq 4$ guarantees that certain algebraic curves have positive genus,
and so have finitely many integer points by Siegel's Theorem.  Moreover, we do not believe that conclusions analogous to those obtained in Theorem~\ref{LehmerGeneral} hold for weight 2 newforms.

\smallskip
\noindent
(v) Some of the results in Theorem~\ref{LehmerGeneral} rely on the GRH.  These cases pertain to situations where GRH was required to  reduce the running time of certain computational number theoretic algorithms. The unconditional bounds lead to infeasible computer calculations.

\end{FiveRemarks}

\begin{example} By Theorem~\ref{LehmerGeneral},  the coefficients of the Hecke eigenform $E_4(z)\Delta(z)$ 
never belong to
$$\{-1\} \cup
\{\pm \ell \ : \  3\leq \ell \leq 37\ {\text {\rm prime}}\}.
$$
Moreover, under GRH the range of the second set can be extended to the odd primes $\ell \leq 97.$
\end{example}

Theorems~\ref{Lehmer135} and \ref{LehmerGeneral} offer variants of Lehmer's speculation for individual newforms.
It is natural to consider an aspect of these questions where the newforms $f$ vary. 
Namely, can a fixed odd $\alpha$ be a Fourier coefficient of
newforms with arbitrarily large weight? We effectively show that this is generically not the case. To ease notation, if $\ell$ is an odd prime, then let
$\mathbb{S}_{\ell}$ denote the set of even weight newforms with integer coefficients,
trivial residual mod 2 Galois representation, and even level that is coprime to $\ell$.

\begin{thm}\label{Power}
If  $m\in \Z^{+},$ then there are effectively computable constants
$M^{\pm}(\ell,m)=O_{\ell}(m)$ for which
$\pm \ell^m$ is not a  coefficient of any
$f\in \mathbb{S}_{\ell}$ with weight $2k>M^{\pm}(\ell,m)$ with $\ell$ ordinary for $f$.
In particular,\footnote{We offer these values to indicate that one can easily work out explicit constants.} 
for $\ell \in \{3, 5\},$ we have
$$
M^{\pm}(\ell,m):= \begin{cases}
2m+10^{23}\sqrt{m} \ \ \ \ \ &{\text {\rm if $\varepsilon=+, m$ odd, and $\ell=3$}},\\
2m+10^{13}\sqrt{m} \ \ \ \ \ &{\text {\rm if $\varepsilon=+, m$ even, and $\ell=3$}},\\
2m+10^{32}\sqrt{m} \ \ \ \ \ &{\text {\rm if $\varepsilon=-$ and $\ell=3$}},\\
3m+10^{24}\sqrt{m} \ \ \ \ \ &{\text {\rm if $\varepsilon=\pm, m$ odd, and $\ell=5$}},\\
3m+10^{13}\sqrt{m} \ \ \ \ \ &{\text {\rm if $\varepsilon=+, m$ even, and $\ell=5$}},\\
3m+10^{30}\sqrt{m} \ \ \ \ \ &{\text {\rm if $\varepsilon=-, m$ even, and $\ell=5$}}.\\
\end{cases}
$$                                  
\end{thm}

\begin{ThreeRemarks}\ \ \ \  \newline 
\noindent
(i) The condition that the level of $f$  is even  is not crucial  for the proof of Theorem~\ref{Power}.
If the level is odd, then  the proof implies that
$a_f(2n+1)\neq \pm \ell^m$ for all $n$ provided that $f$ has  large weight. Furthermore, if $a_f(2)$ is even, then  the stronger claim that $\pm \ell^m$
is not a Fourier coefficient holds.

\smallskip
\noindent
(ii) The condition that the level of $f$ is coprime to $\ell$  also is not crucial. If $\ell$ exactly divides the level, then there is at most one counterexample, and it will be
a Fourier coefficient of the form $a_f(\ell^r)$ (see Theorem~\ref{Newforms} (4)).
Otherwise, the stronger claim holds.

\smallskip
\noindent
(iii) Using the methods in this paper, one can obtain a generalization of Theorem~\ref{Power} for all odd $\alpha$, as well as analogous results for odd weights and forms with real Nebentypus. 
\end{ThreeRemarks}

These results are related to lower
 bounds for the number of prime divisors
of coefficients of newforms. We obtain a general
theorem (see Theorem~\ref{PrimeDivisorsGeneral}) which implies the following lower bound for $\Omega(\tau(n)),$ the number of prime divisors (counted with multiplicity) of $\tau(n)$. As usual,
we let $\omega(n)$ denote the number of distinct prime divisors of $n,$ and we let $\ord_p(n)$ denote the power of $p$ dividing $n.$

\begin{thm}\label{PrimeDivisors}
If $n>1$ is divisible by only ordinary primes, then
$$
\Omega(\tau(n))\geq \sum_{\substack{p\mid n\\ prime}}\left ( \sigma_0(\ord_p(n)+1)-1\right)\geq \omega(n).
$$
\end{thm}

\begin{remark}Theorem~\ref{PrimeDivisors} is sharp, as the
prime in (\ref{LehmerPrime})
satisfies $\Omega(\tau(251^2))=\sigma_0(3)-1=1.$
\end{remark}

The proofs of these results make use of a number of important tools. 
The deep work of Bilu, Hanrot, and Voutier
\cite{BHV} on primitive prime divisors of Lucas sequences forms the primary framework for these results. 
The theory for Lucas sequences applies to the recursion relations given
by Hecke operators in the theory of modular forms. Their work, combined with some combinatorial facts and properties of 2-adic modular Galois representations, leads to Theorems~\ref{PrimeDivisors} and \ref{PrimeDivisorsGeneral}. 
Theorems~\ref{LehmerVariation} and \ref{LehmerVariantGeneral}
follow easily from these results, and they 
 offer an algorithm for locating $\pm \ell^m,$ for odd primes $\ell,$ in Fourier expansions in suitable newforms.
Such occurrences correspond to  special integer points (if any) on elliptic curves, hyperelliptic curves, and certain Thue equations.
 In Section~\ref{IntegerPoints} we classify the integer points on the six curves
in (\ref{HyperEquations}) when $m=1$ (among others), using facts about the classical Lucas sequence,  the Chabauty--Coleman method, and results on Thue equations. We rely  heavily on previous work of Barros \cite{Barros}, Cohn \cite{Cohn}, Bugeaud, Mignotte, and Siksek \cite{BMS}.
With some assistance from Ramanujan's congruences for $\tau(n)$, this classification gives Theorem~\ref{Lehmer135}. In general, this classification
  leads to the proof of Theorem~\ref{LehmerGeneral}.
Finally, in the last section we prove Theorem~\ref{Power} on variants of Lehmer's speculation for large weight newforms.
  
 \section*{Acknowledgements} 
\noindent
The authors thank Malik Amir, Matthew Bisatt, Michael Griffin, Guillaume Hanrot, Vanshika Jain, Sachi Hashimoto, C\'eline Maistret, Drew Sutherland, and Charlotte Ure for their helpful comments during the preparation of this paper. The authors are particularly grateful to Guillaume Hanrot, who
offered assistance with various computer calculations. Finally, we thank the referees for offering further suggestions that improved this paper.

\section{Lucas sequences and the proof of Theorem~\ref{PrimeDivisors}}\label{PPD}

We recall work of Bilu, Hanrot, and Voutier \cite{BHV} on Lucas sequences.
Combining their results with facts about newforms gives Theorem~\ref{PrimeDivisorsGeneral}, which in turn
implies Theorem~\ref{PrimeDivisors}.

\subsection{Lucas sequences and their prime divisors}

Suppose that $\alpha$ and $\beta$ are algebraic integers for which $\alpha+\beta$ and $\alpha \beta$
are relatively prime non-zero integers, where $\alpha/\beta$ is not a root of unity.
Their {\it Lucas numbers} $\{u_n(\alpha,\beta)\}=\{u_1=1, u_2=\alpha+\beta,\dots\}$ are the integers
\begin{equation}
u_n(\alpha,\beta):=\frac{\alpha^n-\beta^n}{\alpha-\beta}.
\end{equation}
A prime  $\ell \mid u_{n}(\alpha,\beta)$ is a {\it primitive prime divisor of $u_n(\alpha,\beta)$} if $\ell \nmid (\alpha-\beta)^2 u_1(\alpha,\beta)\cdots u_{n-1}(\alpha, \beta)$.
Bilu, Hanrot, and Voutier \cite{BHV} proved the following definitive theorem. 

\begin{thm}\label{Bilu}
Every Lucas number $u_n(\alpha,\beta)$, with $n>30,$
has a primitive prime divisor.
\end{thm}

This theorem is sharp; there are sequences for which $u_{30}(\alpha,\beta)$
does not have a primitive prime divisor. 
We call a Lucas number $u_n(\alpha,\beta)$, with $n>2,$ {\it defective}\footnote{We do not consider the absence of
a primitive prime divisor for $u_2(\alpha,\beta)=\alpha+\beta$ to be   a defect.}  if $u_{n}(\alpha,\beta)$ does 
not have a primitive prime divisor. Bilu, Hanrot and Voutier essentially complete the theory; they basically
characterized all of the defective Lucas numbers.
 Their work, combined with a subsequent paper\footnote{This paper included a few cases which were omitted in \cite{BHV}.} 
by Abouzaid \cite{Abouzaid},   gives the {\it complete classification} of
defective Lucas numbers.
Tables 1-4 in Section 1 of \cite{BHV} and Theorem 4.1 of \cite{Abouzaid} offer this
classification. Every defective Lucas number either belongs to a  finite list of sporadic examples or
a finite list of parameterized infinite families.

We consider Lucas sequences  arising from those quadratic integral polynomials
\begin{equation}\label{Modularity}
F(X)=X^2-AX+B=(X-\alpha)(X-\beta),
\end{equation}
where $B=\alpha \beta =p^{2k-1}$ is an odd power of a prime $p$, and $|A|=|\alpha+\beta|\leq 2\sqrt{B}=2p^{\frac{2k-1}{2}}.$ 
 A straightforward analysis of these tables of defective Lucas numbers reveals a list of sporadic examples, and several potentially infinite families of examples.  A straightforward case-by-case analysis using elementary congruences, divisibilities, and the truth of Catalan's conjecture \cite{Catalan},
that $2^3$ and $3^2$ are the only consecutive perfect powers,
yields the following characterization.

\begin{thm}\label{AwesomeList}
Tables \ref{table1} and \ref{table2} in the Appendix list the defective $u_n(\alpha,\beta)$
 satisfying (\ref{Modularity}).
\end{thm}

To identify the cases where $|u_n(\alpha,\beta)| = 1$ and $|u_n(\alpha,\beta)|=\ell$ is prime, we require the curves
\begin{equation}
 B_{1, k}^{r, \pm} : Y^2 = X^{2k-1} \pm 3^r,\quad \mathrm{and}\quad B_{2, k} : Y^2 = 2X^{2k-1} -1.
 \end{equation}

\begin{lem}\label{Part1DefectivePrimality}
Suppose that $u_n(\alpha,\beta)$ is a defective Lucas number from Table \ref{table1} or Table \ref{table2}.
\begin{enumerate}
\item We have that $|u_n(\alpha,\beta)|=1$ if and only if 
$$(A,B,n) \in \big\{ (\pm 1, 2, 5), (\pm 1, 2, 13), (\pm 1, 3, 5), (\pm 1, 5, 7), (\pm 2, 3, 3), (\pm 3, 2^3, 3) \big\},$$
or $(A,B,n) = (\pm m, p, 3),$ where $p = m^2+1$ is prime with $m>1$.

\item If $|u_n(\alpha,\beta)|=\ell$ is prime, then 
$(A,B,\ell, n) \in \big\{ (\pm 1, 2,7, 7), (\pm 1, 2,3, 8), (\pm 2, 11,5, 5) \big\},
$
or $(A,B,\ell, n) = (\pm m, p^{2k-1},3, 3),$ where $(p, \pm m)\in B^{1, \pm}_{1,k}$ and $3\nmid m$,  or $(A,B,\ell, n) = (\pm m, p^{2k-1},m, 4),$ where $(p, \pm m)\in B_{2,k}$.

\end{enumerate}
\end{lem}
\begin{proof} The proof of both (1) and (2) follow by a simple (and tedious) case-by-case analysis. 
\end{proof}

In addition to this classification, we recall several vital facts about Lucas numbers (see Section 2 of  \cite{BHV}).
It is important to know about
their relative divisibility properties. 

\begin{prop}[Prop. 2.1 (ii) of \cite{BHV}]\label{PropA}  If $d\mid n$, then $u_d(\alpha, \beta) | u_n(\alpha,\beta).$
\end{prop}

To keep track of the first occurrence of prime divisors, we let $m_{\ell}(\alpha,\beta)$ be the smallest $n\geq 2$
for which $\ell \mid u_n(\alpha,\beta)$. We note that $m_{\ell}(\alpha,\beta)=2$ if and only if
$\alpha +\beta\equiv 0\pmod {\ell}.$

\begin{prop}[Cor. 2.2\footnote{This corollary is stated for Lehmer numbers. The conclusions hold for Lucas numbers because $\ell \nmid (\alpha+\beta)$.} of \cite{BHV}]\label{PropB} If $\ell\nmid \alpha \beta$ is an odd prime with
$m_{\ell}(\alpha,\beta)>2$, then the following are true.
\begin{enumerate}
\item If $\ell \mid (\alpha-\beta)^2$, then $m_{\ell}(\alpha,\beta)=\ell.$
\item If $\ell \nmid (\alpha-\beta)^2$, then $m_{\ell}(\alpha,\beta) \mid (\ell-1)$ or $m_{\ell}(\alpha,\beta)\mid (\ell+1).$
\end{enumerate}
\end{prop}

\begin{remark}
If $\ell \mid \alpha \beta$, then either $\ell \mid u_n(\alpha,\beta)$ for all $n$, 
or $\ell \nmid u_n(\alpha,\beta)$ for all $n$.
\end{remark}

\subsection{Prime divisors of newform coefficients}\label{SectionPrimeDivisors}

Throughout this paper we suppose that
\begin{equation}\label{qexpansion}
f(z)=q+\sum_{n=2}^{\infty}a_f(n)q^n\in S_{2k}(\Gamma_0(N)) \cap \Z[[q]]
\end{equation}
is an even weight $2k$ newform.
Let $S_f$ be the finite (generally empty) set of primes $p$ for which $(A,B)=(a_f(p),p^{2k-1})$ appears in Tables \ref{table1} or \ref{table2}. 
For primes $p\not \in S_f$ and $m\geq 1$, we let
\begin{equation}
\widehat{\sigma}(p;m):=\sigma_0(m+1)-1,
\end{equation}
while for $p\in S_f$ we define
$\widehat{\sigma}(p;m)$  in Table \ref{table3} in the Appendix.
We have the following theorem.

\begin{thm}\label{PrimeDivisorsGeneral} Assume the notation and hypotheses above.
If $n>1$ is only divisble by ordinary primes, then
$$
\Omega(a_f(n))\geq \sum_{p\mid N} (k-1)\ord_p(n) +\sum_{\substack{p\nmid N \\ \ord_p(n)\geq 2}} \widehat{\sigma}(p;\ord_p(n)).
$$
\end{thm}

\begin{remark} 
Theorem~\ref{PrimeDivisorsGeneral} does not take into account those primes $p\nmid N$ which
exactly divide $n$ because it can happen that $|a_f(p)|=1$. 
However,  if the mod 2 residual Galois representation
is trivial, then $a_f(p)$ is even for every prime $p\nmid 2N$. In such cases, we get
$$
\Omega(a_f(n))\geq \sum_{p\mid N} (k-1)\ord_p(n) +\sum_{p\nmid 2N} \widehat{\sigma}(p;\ord_p(n)).
$$
This  applies to $\Delta(z)$, by the congruence
$\Delta(z)\equiv \sum_{n=0}^{\infty}q^{(2n+1)^2}\pmod 2.$
Since $(A,B)=(\tau(p),p^{11})$ does not appear in Lemma~\ref{Part1DefectivePrimality} (1), the proof of Theorem~\ref{PrimeDivisorsGeneral} gives
Theorem~\ref{PrimeDivisors}.
\end{remark}

\subsection{Proof of Theorem~\ref{PrimeDivisorsGeneral}}

We recall some basic facts about {\it Atkin-Lehner newforms} 
(see \cite{AtkinLehner, Miyake}), along with the deep theorem of Deligne  \cite{Deligne1, Deligne2} that bounds their Fourier coefficients. 

\begin{thm}\label{Newforms} Suppose that $f(z)=q+\sum_{n=2}^{\infty}a_f(n)q^n\in S_{2k}(\Gamma_0(N))$ is a newform with integer coefficients.
Then the following are true:
\begin{enumerate}
\item If $\gcd(n_1,n_2)=1,$ then $a_f(n_1 n_2)=a_f(n_1)a_f(n_2).$
\item If $p\nmid N$ is prime and $m\geq 2$, then
$$
a_f(p^m)=a_f(p)a_f(p^{m-1}) -p^{2k-1}a_f(p^{m-2}).
$$
\item If $p\nmid N$ is prime and $\alpha_p$ and $\beta_p$ are roots of $F_p(x):=x^2-a_f(p)x+p^{2k-1},$ then
$$
   a_f(p^m)=u_{m+1}(\alpha_p,\beta_p)=\frac{\alpha_p^{m+1}-\beta_p^{m+1}}{\alpha_p-\beta_p}.
$$   
Moreover, we have $|a_f(p)|\leq 2p^{\frac{2k-1}{2}}$, and $\alpha_p$ and $\beta_p$ are complex conjugates.

\item If $p\mid N$ is prime, then 
$f | U(p) := \sum_{n=1}^{\infty} a_f(np)q^n = a_f(p) f(\tau).$
Moreover, we have 
$$
a_f(p^m)=\begin{cases} (\pm 1)^m p^{(k-1)m} \ \ \ \ \  &{\text {\rm if}}\ \ord_p(N)=1,\\
0 \ \ \ \ \ &{\text {\rm if}}\ \ord_p(N)\geq 2.
\end{cases}
$$
\end{enumerate}
\end{thm}

Theorem~\ref{Newforms} leads to lower bounds for
 the number of prime divisors (counted with multiplicity) of the coefficients in the sequence $\{a_f(p^2),a_f(p^3),\dots\}$, where $p$ is prime.

\begin{prop}\label{PrimePower} Assuming the notation in Theorem~\ref{Newforms},  the following are true for $m\geq 2$.
\begin{enumerate}
\item If $p\mid N$ is prime, then $\ord_p(a_f(p^m))\geq (k-1)m.$
\item Suppose that $p\nmid N$ is prime. If $(A,B)=(a_f(p),p^{2k-1})$ does not appear in Tables \ref{table1} or \ref{table2}, then
$$
\Omega(a_f(p^m))\geq \sigma_0(m+1)-1.
$$
\item Suppose that $p\nmid N$ is prime. If $(A,B)=(a_f(p),p^{2k-1})$ appears in Tables \ref{table1} or \ref{table2}, then Table \ref{table3} of the Appendix contains a lower bound
for $\Omega(a_f(p^m))$.
\end{enumerate}
\end{prop}

\begin{proof}[Proof of Proposition~\ref{PrimePower}]
The first claim follows from Theorem~\ref{Newforms} (4). The second claim follows from Theorem~\ref{Newforms} (3), 
Proposition~\ref{PropA} 
and Theorem~\ref{Bilu} in a case-by-case analysis. The point is that at least one new prime divisor is accumulated with each subsequent step in a Lucas sequence. In other words,  the relative divisibility of Lucas numbers and the presence of primitive prime divisors guarantees the lower bound. The only divisor of $m+1$ which does not contribute is $u_1=1$.
The third claim follows similarly by taking into account the  defective Lucas numbers
that appear in Tables \ref{table1} and \ref{table2}.
\end{proof}

\begin{proof}[Proof of Theorem~\ref{PrimeDivisorsGeneral}]
The theorem follows from Theorem~\ref{Newforms} (1)  and Proposition~\ref{PrimePower}.
\end{proof}

\section{Variations of Lehmer's Speculation}\label{SectionLehmer}

Regarding coefficients of newforms satisfying (\ref{qexpansion}), we classify those $n$ for which
$|a_f(n)|=\ell$ is an odd prime. For the remainder of the paper, we assume that all newforms have weight $2k\geq 4$.
We first determine when $|a_f(n)|=1$. Define the set
\begin{equation}
\mathcal{U}_f:=\begin{cases} \{1, 4\}  \  \ \ \ \ &{\text {\rm if}}\ a_f(2)=\pm 3,\ 2k=4, {\text {\rm and}}\ N\ {\text {\rm odd}}\},\\
\{1\} \ \ \ \ &{\text {\rm otherwise.}}
\end{cases}
\end{equation}
\begin{prop}\label{One}
Suppose that the mod 2 residual Galois representation for $f(z)$ is trivial. Then we have
 $|a_f(n)|=1$ if and only if $n\in \mathcal{U}_f.$
\end{prop}
\begin{proof}
By multiplicativity (i.e. Theorem~\ref{Newforms} (1)), it suffices to
determine when $|a_f(p^m)|=1$, where $p$ is prime.
By Proposition~\ref{PrimePower} (1), we have $p\nmid N.$ By Theorem~\ref{Newforms}~(3), it suffices to determine when
the $|u_{m+1}(\alpha_p,\beta_p)|=1,$ where $m\geq 2.$ Indeed, $a_f(p)=u_2(\alpha_p,\beta_p)$ is even for  $p\nmid 2N$.
By Theorem~\ref{Bilu}, this reduces to
Lemma~\ref{Part1DefectivePrimality} (1).
The defective cases $(A,B,n)=(\pm 3, 2^3,3)$  correspond to potential weight 4 newforms, while the remaining possibilities
are for weight 2. In the weight 4 cases we have $a_f(2)=\pm 3$, which  gives $a_f(4)=a_f(2)^2-2^3=1.$
\end{proof}

\begin{thm}\label{LehmerVariantGeneral}
Suppose that the mod 2 residual Galois representation for $f(z)$ is trivial, and that $\ell \nmid a_f(\ell).$ If $|a_f(n)|=\ell^m,$  with $m\in \Z^{+}$ and $\ell$ is an odd prime, then 
$n= m_0p^{d-1}$, where $m_0\in \mathcal{U}_f$, $p\nmid N$ is prime, and $d \mid \ell (\ell^2-1)$
 is an odd prime.
 Moreover, $|a_f(n)|=\ell^m$ for finitely many (if any) $n$. 
\end{thm}

\begin{proof}[Proof of Theorem~\ref{LehmerVariation} and \ref{LehmerVariantGeneral}]
By Proposition~\ref{One} and Theorem~\ref{Newforms} (1) and (4), it suffices to determine when
$|a_f(p^{d-1})|=|u_{d}(\alpha_p,\beta_p)|=\ell$, where $p\nmid N$ is prime. Since  $2k\geq 4,$
$\ell$ is odd, and $A=a_f(p)$ is even, Lemma~\ref{Part1DefectivePrimality} (2) leaves the defective possibilities
 $(A,B,\ell,n)=(\pm m,p^{2k-1},3,3)$, which by Theorem~\ref{Newforms} (2), implies that $(p,a_f(p))$ is an integer point on
$Y^2=X^{2k-1}\pm 3.$
This means that $u_3(\alpha_p,\beta_p)=a_f(p^2)=\pm 3$, which 
is the claimed conclusion with $d=\ell=3$.

Now we consider whether a prime power can be
a nondefective Lucas number $u_{d}(\alpha_p,\beta_p)=a_f(p^{d-1})$, for primes $p\nmid 2N$.
Since $a_f(p)$ is even, we may assume that $\ell \nmid \alpha_p \beta_p$ and $m_{\ell}(\alpha_p,\beta_p)>2$. Moreover, Theorem~\ref{Newforms} (2) implies
that $a_f(p^b)$ is odd if and only if $b$ is even, and so we may assume that $d$ is odd.
Proposition~\ref{PropB} implies that $m_{\ell}(\alpha_p,\beta_p)=\ell$ or $m_{\ell}(\alpha_p,\beta_p) | (\ell-1)$ or
$m_{\ell}(\alpha_p,\beta_p) | (\ell+1)$.

Due to the generic presence of primitive prime divisors, a Lucas number that is a prime power $\ell^m$ in absolute
value is the first multiple of $\ell$ in the sequence.
By Theorem~\ref{Bilu}, Proposition~\ref{PropA}, and Lemma~\ref{Part1DefectivePrimality} (2), this
holds for every sequence satisfying (\ref{Modularity}) for weights $2k\geq 4$.
In particular, $d$ is an odd prime. 
The finiteness of the number of $p$ for which $|a_f(p^{d-1})|=\ell$, follows from Siegel's Theorem, that positive genus curves  have at most finitely many integer points. These curves are easily assembled using Theorem~\ref{Newforms} (2) (see Lemma~\ref{DiophantineCriterion}).
\end{proof}

\section{Integral Points on some curves}\label{IntegerPoints}
To prove Theorems~\ref{Lehmer135} and \ref{LehmerGeneral}, we require knowledge of the integer points on certain curves.

\subsection{Some Thue equations}
An equation of the form
$F(X,Y)=D,
$
where $F(X,Y)\in \Z[X,Y]$ is homogeneous and $D$ is a non-zero integer, is known as a {\it Thue equation}.
We require such equations that arise
from the generating function
\begin{equation}\label{genfunction}
\frac{1}{1-\sqrt{Y}T+XT^2}=\sum_{m=0}^{\infty}F_m(X,Y)\cdot T^{m}=1+\sqrt{Y}\cdot T+(Y-X)T^2+\cdots.
\end{equation}
The first few homogenous polynomials $F_{2m}(X,Y)$ are as follows:
\begin{displaymath}
\begin{split}
F_2(X,Y)&=Y-X,\\
F_4(X,Y)&=Y^2-3XY+X^2\\
F_6(X,Y)&=Y^3-5XY^2+6X^2Y-X^3.\\
F_{10}(X,Y)&=Y^5-9XY^4+28X^2Y^3-35X^3Y^2+15X^4Y-X^5.
\end{split}
\end{displaymath}
For every positive integer $m$, we consider the degree $m$ Thue equations of the form
\begin{equation}\label{ThueEqn}
F_{2m}(X,Y)=\prod_{k=1}^m \left(Y-4X \cos^2\left(\frac{\pi k}{2m+1}\right)\right)=D.
\end{equation}

The next lemma gives integer points on several Thue equations that 
we shall require.

\begin{lem}\label{Monster} The following are true.
\begin{enumerate}
\item
Table \ref{thuetable} in the Appendix lists all of the integer solutions to
$$F_{d-1}(X,Y)=\pm\ell
$$
for every pair of odd primes $(d,\ell)$ for which $7\leq d\mid \ell(\ell^2-1)$ and $\ell \in \{ 7\leq \ell \leq 37\}$.
\item Conditional on GRH, Table \ref{thueGRHtable} in the Appendix lists all of the integer solutions to
$$
F_{d-1}(X,Y)=\pm\ell
$$
for every pair of odd primes $(d,\ell)$ for which $7\leq d\mid \ell(\ell^2-1)$ and $41\leq \ell \leq 97.$
\item There are no integer solutions to
$F_{22}(X,Y)=\pm 691.$
\item The points $(\pm 1, \pm 4)$ are the only integer solutions to
$F_{690}(X,Y)=\pm 691.$
\end{enumerate}
\end{lem}

\begin{proof} Claims (1), (2) and (3) are easily obtained using the Thue solver in \texttt{PARI/GP} \cite{pari}
(see \cite{github} for all of the code required for this paper).

The proof of (4) is more formidable, as $F_{690}(X,Y)$ has degree 345. However, for odd primes $p$, the Thue equations $F_{p-1}(X,Y)=\pm p$
 are equivalent to the well-studied equations
\begin{equation}\label{ModifiedThue}
\widehat{F}_p(X,Y)=\prod_{k=1}^{\frac{p-1}{2}}\left(Y-2X\cos\left(\frac{2\pi k}{p}\right)\right)=\pm p
\end{equation}
that were prominent in the work of Bilu, Hanrot, and Voutier on primitive prime divisors of Lucas sequences. 
Indeed, we have  $F_{p-1}(X,Y)=\widehat{F}_p(X,Y-2X).$
They prove the important  fact
(see Cor. 6.6 of \cite{BHV}) that there are no integer solutions to (\ref{ModifiedThue})
with $|X|>e^8$  when $31\leq p\leq 787.$ By a well-known criterion (for example, see Lemma~1.1 of \cite{TW} and Proposition 2.2.1 of \cite{BH96})), midsize solutions of $\widehat{F}_{691}(X,Y)=\pm 691$  correspond to convergents of the continued fraction expansion of some 
$2\cos(2\pi k/691).$  A short calculation rules this out, possibly leaving some small solutions, those  with
 $|X|\leq 4$.  For these $X$, we find $(\pm 1, \pm 2)$, which implies that
$(\pm 1, \pm 4)$ are the only integral solutions to $F_{690}(X,Y)=\pm 691.$
\end{proof}

\subsection{The elliptic and hyperelliptic curves $Y^2=X^{2d-1}\pm \ell$}

For  $d\in \{2, 3, 4 ,6, 7\}$ and odd primes $\ell  \leq 97$, we list all of the integer points on
\begin{equation}
C_{d,\ell}^{\pm}: Y^2=X^{2d-1}\pm \ell.
\end{equation}

\begin{lem}\label{HYPER} If $3\leq \ell \leq 97$ is prime and $d\in \{2, 3, 4, 6, 7\}$, then the following are true:
\begin{enumerate}
\item Table \ref{Cplustable} in the Appendix lists the integer points on $C_{d,\ell}^{+}.$
\item Table \ref{Cminustable} in the Appendix lists the integer points on $C_{d,\ell}^{-}.$
 \end{enumerate}
\end{lem}
\begin{proof}
Work by Barros \cite{Barros}, Cohn \cite{Cohn} and Bugeaud, Mignotte and Siksek \cite{BMS}
establish these claims. Table \ref{Cplustable} is assembled from the Appendix of \cite{Barros}, and Table \ref{Cminustable} is assembled from the Appendix of \cite{BMS}.
\end{proof}

\subsection{The hyperelliptic curves $Y^2=5X^{2d}\pm 4\ell$}

For $d\geq 2,$ we define the hyperelliptic curves
\begin{equation}
H^{\pm}_{d,\ell}:  Y^2 =5 X^{2d}\pm 4\ell.
\end{equation}
The following satisfying lemma classifies the integer points on $H_{d,5}^{\pm}.$

\begin{lem}\label{AnnalsCorollary}
If $\ell=5$, then the following are true.
\begin{enumerate}
\item If $d=2$ and $\ell=5$, then the only integer points on $H^{+}_{2,5}$ are $(\pm 1, \pm 5)$ and $(\pm 2, \pm 10)$.
\item If $d>2,$ then the only integer points on $H^{+}_{d,5}$ are $(\pm 1, \pm 5).$
\item If $d\geq 2,$ then $H^{-}_{d,5}$ has no integer points.
\end{enumerate}
\end{lem}
\begin{proof}
We recall the classical Lucas sequence $$\{L_n\}=\{2, 1, 3, 4, 7, 11, 18, 29, 47,76, 123, 199, 322, 521, 843,\dots\},$$ defined by
$L_0:=2$ and $L_1:=1$ and the recurrence $L_{n+2}:=L_{n+1}+L_n$ for $n\geq 0$.
A  theorem of Bugeaud, Mignotte, and Siksek \cite{AnnalsFibonacci} asserts that $L_1=1$ and $L_3=4$
are the only perfect power Lucas numbers.
By the theory of Pell's equations, the positive integer $X$-coordinate solutions to
$H^+_{1,5}$ and $H^{-}_{1,5},$
namely $\{L_1=1, L_3= 4,L_5=11,\dots\}$ and
$\{L_0=2, L_2=3, L_4=7,\dots\}$ respectively, split the Lucas numbers. The three claims follow immediately.
\end{proof}

For  primes $\ell \in \{691\}\cup\left \{ 11\leq \ell\leq 89 \ : \  {\text {\rm prime with }} \leg{\ell}{5}=1
\right\}$, we have the following lemma.

\begin{lem}\label{Hyper11_19} The following are true.
\begin{enumerate}
\item
For most\footnote{We were unable to obtain results for $H_{7,71}^{+},$
$H_{13,89}^{-},$ and any $H^{+}_{11,\ell}$ and $H^{+}_{13,\ell}.$}
$d\in \{3, 5, 7, 11, 13\}$ and primes $\ell \in \left\{ 11\leq \ell \leq 89\ : \ \leg{\ell}{5}=1\right\}$,
Table \ref{Htable} in the Appendix lists (some cases conditional on GRH) the integer points on $H^{\pm}_{d,\ell}.$
\item There are no integer points on $C^{-}_{6,691}.$
\item There are no integer points on $H^{-}_{11,691}.$
\end{enumerate}
\end{lem}
\begin{proof}
Generalized Lebesgue--Ramanujan--Nagell equations are  equations of the form
\begin{equation}\label{GLRN}
x^2+D=Cy^n,
\end{equation}
where $D$ and $C$ are non-zero integers. An integer point on (\ref{GLRN}) can be studied in the ring of integers
of $\Q(\sqrt{-D})$
using the factorization
$$(x+\sqrt{-D})(x-\sqrt{-D}) =Cy^n.
$$

This observation is a standard tool in the study of Thue equations. In particular, 
Theorem~2.1 of  \cite{Barros} (also see Proposition~3.1 of \cite{BMS}) gives a step-by-step algorithm that takes alleged solutions of (\ref{GLRN}) and produces integer points on one
of finitely many Thue equations constructed from $C, D$ and $n$ via the algebraic number theory of $\Q(\sqrt{-D})$.
 These equations are assembled from the knowledge of the group of units and the
ideal class group.

To prove all three parts of the lemma (apart from $H_{7,89}^{+}$), we implemented this algorithm in \texttt{SageMath} (see \cite{github} for all \texttt{SageMath} code required for this paper). 
Some cases required GRH as a simplifying assumption. As the curves in (2) and (3) are the most complicated, we offer brief details in these two cases.

To prove (2), we consider the  hyperelliptic curve $C^{-}_{6,691},$ which  corresponds to (\ref{GLRN}) for the class number 5 imaginary quadratic field
$\Q(\sqrt{-691})$, where $x=Y, y=X, C=1, D=691,$ and $n=11.$  In this case the algorithm gives exactly one Thue equation, which after clearing denominators can be rewritten as
\small
\begin{displaymath}
\begin{split}
2\times5^{55}&=(991077174272090396)x^{11} + (119700018439220789119)x^{10}y
- (8831599221002836172345)x^9y^2\\
&\ \ \ \ -(337116345512786456280840)x^8y^3 
+ (8492967300375371034332430)x^7y^4\\
&\ \ \ \  + (175189311986919278870504298)x^6y^5 
- (1881807368163995585644810248)x^5y^6\\
&\ \ \ \  - (22992541672786450593030038430)x^4y^7 
+ (104772541553739359102253613965)x^3y^8\\
&\ \ \ \  + (697875798749922445133117312720)x^2y^9 
- (1068801486169809452619368218519)xy^{10}\\
&\ \ \ \  - (2292300374810647823111384294421)y^{11}.
\end{split}
\end{displaymath}
\normalsize
The Thue equation solver in \texttt{PARI/GP}, which implements the Bilu--Hanrot algorithm, establishes that there are no integer solutions, and so $C^{-}_{6,691}$ has no integer points.

Claim (3) is about the hyperelliptic curve $H^{-}_{11,691}.$  Its integer points $(X,Y)$  satisfy
$$
(Y+2\sqrt{-691})(Y-2\sqrt{-691})=5X^{22}.
$$
Therefore, we again employ the imaginary quadratic field $\Q(\sqrt{-691})$. In particular, we have (\ref{GLRN}), where $x=Y, y=X, C=5, D=4\cdot 691$ and $n=22$.
 The algorithm again gives one Thue equation, which after clearing denominators can be rewritten as

\small 
\begin{displaymath}
\begin{split}
2^2\times5^{110}&=
-(20587212586465949627980680671826599752) x^{22} \\
&\ \  \ \   + (1133274396835827658613802749227310922394) x^{21} y\\ 
&\ \ \ \ +\cdots\\
 &\ \ \ \  -(79670423145107301772779399379735976309907264511718034789276856) x y^{21}\\
&\ \ \ \  +(71809437208138431262783549625248617351731199323326115439324273)y^{22}.
 \end{split}
 \end{displaymath}
\normalsize
The Thue solver in \texttt{PARI/GP} establishes that there are no integer solutions, and so $H^{-}_{11,691}$ has no integer points.
\end{proof}

We use the Chabauty--Coleman method\footnote{We could have (in theory) used the Thue method as in the proof of Lemma~\ref{Hyper11_19}. We chose this method as it did not require substantial computer resources.}, which employs $p$-adic integration to determine the rational points on suitable curves of genus $g\geq 2,$  to determine  the integer points on  $C_{6,691}^{+}$, $H_{7,89}^{+}$, and
$H_{11,691}^{+}.$

\begin{lem}\label{Plus691} The following are true.
\begin{enumerate}
\item There are no integer points on $C^{+}_{6,691}.$
\item There are no integer points on  $H^{+}_{11,691}.$
\item Assuming GRH, the only integer points on $H^{+}_{7,89}$ have $(|X|,|Y|)=(1,19).$
\end{enumerate}
\end{lem}
\begin{proof}
We employ the Chabauty--Coleman method \cite{coleman} to determine the integral points on these curves. 

We first prove (1). The genus 5 curve $C^{+}_{6,691}$ has Jacobian with Mordell-Weil rank 0. This can be determined using the implementation of 2-descent in \texttt{Magma} \cite{magma}. Since the rank is less than the genus,  the Chabauty--Coleman method applies, which, in this case, gives a 5-dimensional space of regular 1-forms vanishing on rational points. We take as our basis for the space of annihilating differentials the set $\{\omega_i := X^i \frac{dX}{2Y}\}_{i = 0, 1, \ldots, 4}.$ The prime $p = 3$ is a prime of good reduction for $C^{+}_{6,691}$, and taking the point at infinity $\infty$ as our basepoint, we compute the set of points $$\left\{z \in C^{+}_{6,691}(\Z_3): \int_{\infty}^z \omega_i = 0\;\textrm{for all}\;{i = 0, 1, \ldots, 4}\right\},$$ where the integrals are Coleman integrals computed using \texttt{SageMath} \cite{sage}. By construction,  this set contains the integral points on the working affine model of $C^{+}_{6,691}$. 

The computation gives three points: two points with $X$-coordinate 0 and a third point with $Y$-coordinate 0 in the residue disk corresponding to $(2,0) \in C^{+}_{6,691}(\F_3)$. (Indeed, the power series corresponding to the expansion of the integral of $\omega_0$ has each of these points occurring as simple zeros.)  Hence, there are no integral points on $C^{+}_{6,691}$.

Turning to $H^{+}_{11,691}$, we consider the integral points on the curve $Y^2 = 5X^{11} + 4 \cdot 691$ and then pull back any points found using the map $(X,Y) \rightarrow (X^2,Y)$. Using \texttt{Magma}, we find that the rank of the Jacobian of this genus 5 curve is 0. We rescale variables to work with the monic model $Y^2 = X^{11}+4\cdot 5^{10}\cdot 691$ and we apply the Chabauty--Coleman method using $p = 3$. As before, the computation gives three points with coordinates in $\Z_3$: two points with $X$-coordinate 0 and a third point with $Y$-coordinate 0 in the residue disk corresponding to $(2,0)$. The power series corresponding to the expansion of the integral of $\omega_0$ has each of these points occurring as simple zeros. None of these points are rational. Therefore, $H^{+}_{11,691}$ has no integral points. This proves (2).

Now we turn to (3). To compute integral points on $H^{+}_{7,89}$, we work with the genus 3 curve $Y^2 = 5X^{7} + 4 \cdot 89$ and then pull back any integral points found using the map $(X,Y) \rightarrow (X^2,Y)$. Using \texttt{Magma}, we find that the rank of the Jacobian of this genus 3 curve is 2, under the assumption of GRH\footnote{The \texttt{Magma} procedure that computes ranks requires GRH in this case to be computationally feasible.}. We work with the monic model $$H_m: Y^2 = X^{7}+4\cdot 5^{6}\cdot 89$$ and run the Chabauty--Coleman method using $p = 3$.

The points 
$$P = [x^3 + 14x^2 - 800, 9x^2 + 200x - 4050] \qquad\textrm{and}\qquad Q = [x-5, 19 \cdot 5^3]$$
(given in Mumford representation) are independent in the Jacobian of $H_m$. To simplify the Chabauty--Coleman computation---in particular, so that we carry out all of our computations over $\Q_3$---we replace $P$ with $P'$, a small $\Z$-linear combination of $P$ and $Q$ that is linearly independent from $Q$, with the property that the first coordinate of the Mumford representation of $P'$ splits over $\Q_3$. 

We take $P' := 2P - 5Q$, with Mumford representation of $P'$ given by $[f(x),g(x)]$ where
\tiny{
\begin{align*}
f(x) &= x^3 - \frac{57819608106819190393450758001494220029312032281}{243432625872206959773347921129373894485149809}x^2 +\frac{301022057022978383553067428985393708004188803800}{81144208624068986591115973709791298161716603}x -\\
 &\qquad   
              \frac{4935244227803215636634926465657011220846146763100}{243432625872206959773347921129373894485149809}, \\
g(x)&= \frac{13467788979408324218581419111573847035681150845619031139253274307312471}{3798115572194618764136691476777323149900556269646219373513689210377}x^2 -\\  
   &\qquad            \frac{73837091689655128840131596065726589815272462202819205672839132728899500}{1266038524064872921378897158925774383300185423215406457837896403459}x +\\
 &\qquad  \frac{1249983247105360333943070938652709476597593148217064351317870016169354850}{3798115572194618764136691476777323149900556269646219373513689210377}.
\end{align*}
}
\normalsize
To compute an annihilating differential, we compute the $3\times 2$ matrix of Coleman integrals $(\int_{P'} \omega_i, \int_Q \omega_i)_{i = 0, 1, 2}$, where $\omega_i = X^i \frac{dX}{2Y}$, in \texttt{Sage}:\tiny{
$$\left(\begin{array}{rr}
2 \cdot 3 + 2 \cdot 3^{2} + 3^{4} + 2 \cdot 3^{6} + 3^{8} + 2 \cdot 3^{9} + O(3^{10}) & 3^{3} + 2 \cdot 3^{4} + 3^{7} + 2 \cdot 3^{8} + 3^{9} + O(3^{10}) \\
2 \cdot 3 + 3^{2} + 3^{3} + 2 \cdot 3^{5} + 2 \cdot 3^{6} + 2 \cdot 3^{7} + O(3^{10}) & 2 \cdot 3 + 3^{2} + 3^{3} + 2 \cdot 3^{7} + 2 \cdot 3^{8} + 3^{9} + O(3^{10}) \\
3 + 3^{2} + 2 \cdot 3^{3} + 2 \cdot 3^{4} + 2 \cdot 3^{5} + 3^{6} + 3^{7} + 2 \cdot 3^{9} + O(3^{10}) & 2 \cdot 3 + 3^{2} + 3^{3} + 2 \cdot 3^{4} + 3^{5} + 3^{7} + 2 \cdot 3^{8} + 2 \cdot 3^{9} + O(3^{10})\end{array}\right).$$}

\normalsize

We then compute a basis of the kernel of this matrix, which gives us our annihilating differential 
\begin{align*}\omega &= \omega_0 + (1 + 2 \cdot 3^{2} + 2 \cdot 3^{4} + 3^{5} + 3^{6} + 2 \cdot 3^{7} + 2 \cdot 3^{8} + 2 \cdot 3^{9} +  O(3^{10})) \omega_1 \\
&\qquad\; + (2 + 2 \cdot 3 + 3^{2} + 3^{3} + 2 \cdot 3^{4} + 3^{5} + 2 \cdot 3^{6} + 3^{9} + O(3^{10}))\omega_2.
\end{align*}

Finally, we have three residue disks to consider, corresponding to $(1,0)$ and $(2, \pm 1) \in H_{m}(\F_3)$. We compute the set of points $z \in H_m(\Z_3)$ in these residue disks such that $\int_{\infty}^z \omega = 0$. This produces three points, each occurring as simple zeros of the corresponding $3$-adic power series: a Weierstrass point and the points $(5, \pm 2375).$
The Weierstrass point is not rational, while the points $(5, \pm 2375)$ correspond to the points $( \pm 1, \pm 19)$ on $H^{+}_{7,89}$.
\end{proof}

\section{Proofs of Theorems~\ref{Lehmer135} and \ref{LehmerGeneral}}

We combine  results from the previous section with
Theorem~\ref{LehmerVariantGeneral} to prove Theorems~\ref{Lehmer135} and \ref{LehmerGeneral}.
The following lemma, which relates Fourier coefficients to special integer points on algebraic curves, is a straightforward consequence of Theorem~\ref{Newforms} (2) and (3).

\begin{lem}\label{DiophantineCriterion}
Assuming the notation in Theorem~\ref{Newforms}, if $p\nmid N$ is prime, then
we have the following:
\begin{enumerate}
\item If $a_f(p^2)=\alpha$, then $(p,a_f(p))$ is an integer point on
$$
Y^2=X^{2k-1}+\alpha.
$$
\item If $a_f(p^4)=\alpha$, then $(p,2a_f(p)^2-3p^{2k-1})$ is an integer point on
$$
Y^2=5X^{2(2k-1)}+4\alpha.
$$
\item For every positive integer $m$ we have that
$F_{2m}(p^{2k-1},a_f(p)^2)=a_f(p^{2m}).$
\end{enumerate}
\end{lem}

\begin{proof}[Proof of Theorem~\ref{Lehmer135}]
It is well-known that $\tau(n)$ is odd if and only if $n$ is an odd square. 
To see this, we employ 
 the Jacobi Triple Product identity to obtain the congruence
\begin{displaymath}
\begin{split}
\sum_{n=1}^{\infty}\tau(n)q^n:&=q\prod_{n=1}^{\infty}(1-q^n)^{24}
\equiv q\prod_{n=1}^{\infty}(1-q^{8n})^3
=\sum_{k=0}^{\infty} (-1)^k(2k+1)q^{(2k+1)^2}\pmod 2.
\end{split}
\end{displaymath}

We consider the possibility that $\pm 1$ appear in sequences of the form
\begin{equation}\label{primepowers}
\{\tau(p),\tau(p^2), \tau(p^3),\dots\}.
\end{equation}
By Theorem~\ref{Newforms} (2), if $p$ is prime and  $p\mid \tau(p)$, then $p^m\mid \tau(p^m)$ for every $m\geq 1$, and
so $|\tau(p^m)|\neq 1.$  Moreover, $|\tau(p)|\neq p$, where $p$ is an odd prime, because $\tau(p)$ is even.
Therefore, such sequences may be completely ignored for the remainder of the proof.

For primes
$p\nmid \tau(p),$ Theorem~\ref{Newforms} (3) gives a Lucas sequence 
with $A=\tau(p)$ and $B=p^{11}.$
Lemma~\ref{Part1DefectivePrimality} shows that there are no defective terms with
$u_{m+1}(\alpha_p,\beta_p)=\tau(p^m)\neq \pm 1$ or $\pm \ell$, where $\ell$ is an odd prime.
To see this, we note that $A=\tau(p)$ is even.  Lemma~\ref{Part1DefectivePrimality} (2) does not allow
for $A$ to be even with one exception, the possibility that  $(A,B,\ell,n )=(\pm m,p^{11}, 3,3)$, where $(p,\pm m)\in B_{1,6}^{1,\pm}.$
However, these curves are the same as $C_{6,3}^{\pm},$ and Lemma~\ref{HYPER}  shows that there are no such points.
Therefore, we may assume that all of 
the values in (\ref{primepowers}) have a primitive prime divisor, and  never have  absolute value 1.

We now turn to the primality of absolute values of $\tau(n)$.
Thanks to Hecke multiplicativity (i.e. Theorem~\ref{Newforms} (1)) and the discussion above, if $\ell$ is an odd prime and $|\tau(n)|=\ell$, then $n=p^d$, where $p$ is an odd prime for which
 $p\nmid \tau(p).$ The fact that $\tau(p^d)=u_{d+1}(\alpha_p,\beta_p)$ leads to a further constraint on $d$ 
(i.e. refining 
 the fact that $d$ is even).
By Proposition~\ref{PropA}, which guarantees relative divisibility between Lucas numbers, and 
 Lemma~\ref{Modularity}, which guarantees the absence of defective terms in (\ref{primepowers}), it follows that 
$d+1$ must be an odd prime, and $\tau(p^d)$ is the very first term that is divisble by $\ell$.
 
To make use of this observation, for odd primes 
$p$ and $\ell$ we define
\begin{equation}
m_{\ell}(p):=\min\{ n\geq 1\ : \ \tau(p^n)\equiv 0\!\!\!\!\pmod{\ell}\}.
\end{equation}
For $|\tau(p^d)|=\ell$, we have $m_{\ell}(p)=d,$ where $d+1$ is also an odd prime.
The Ramanujan congruences \cite{RamanujanUnpublished, Ramanujan, SerreRamanujan}
\begin{displaymath}
\tau(n)\equiv \begin{cases} &n^2\sigma_1(n)\pmod 9,\\
 &n\sigma_1(n)\pmod 5,\\
&n\sigma_3(n)\pmod 7,\\
&\sigma_{11}(n)\pmod{691},
\end{cases}
\end{displaymath}
where $\sigma_{\nu}(n):=\sum_{1\leq d\mid n}d^{\nu}$, make it simple to compute $m_{\ell}(p)$
for the primes $\ell \in \{3, 5, 7, 691\}.$

Thanks to the mod 9 congruence, we find that
$$
m_3(p)=\begin{cases} 1 \ \ \ \ \ &{\text {\rm if }} p\equiv 0, 2\!\!\!\!\pmod 3,\\
2\ \ \ \ \ &{\text {\rm if }} p\equiv 1\!\!\!\!\pmod 3.
\end{cases}
$$
Therefore, $d=2$ is the only possibility.
If $\tau(p^2)=\pm 3$, then Lemma~\ref{DiophantineCriterion} (1) implies that
 $(p,\tau(p))$ is a point on $C^{\pm}_{6,3},$ which were considered immediately above.
Again,  Lemma~\ref{HYPER} (1) implies that there are no such integer points.
 
Thanks to the mod 5 congruence, we find that
$$
m_5(p)=\begin{cases} 1 \ \ \ \ \ &{\text {\rm if }} p\equiv 0, 4\!\!\!\!\pmod{5},\\
3 \ \ \ \ \ &{\text {\rm if }} p\equiv 2, 3\!\!\!\!\pmod{5},\\
4 \ \ \ \ \ &{\text {\rm if }} p\equiv 1\!\!\!\!\pmod{5}.
\end{cases}
$$
Therefore, $d=4$ is the only possibility.
If $\tau(p^4)=\pm 5$, then Lemma~\ref{DiophantineCriterion} (2) implies that
 $(p, 2\tau(p)^2-3p^{11})$ is an integer point on $H^{\pm}_{11,5}.$
Lemma~\ref{AnnalsCorollary} shows that no such points exist on these hyperelliptic curves.

Thanks to the mod 7 congruence, we find that
$$
m_7(p)=\begin{cases}
1 \ \ \ \ \ &{\text {\rm if }} p\equiv 0, 3, 5, 6\!\!\!\!\pmod{7},\\
6 \ \ \ \ \ &{\text {\rm if }} p\equiv 1, 2, 4\!\!\!\!\pmod{7}.
\end{cases}
$$
Hence, $d=6$ is the only possibility, and so we must rule out the possibility that $\tau(p^6)=\pm 7$.
If there are such primes $p$, then Lemma~\ref{DiophantineCriterion} (3) implies
 that
$F_{6}(p^{11},\tau(p)^2)=\pm 7.$
Lemma~\ref{Monster} (1) shows that there are no such solutions to $F_6(X,Y)=\pm 7.$

Thanks to the mod 691 congruence, we find that
the only cases where $m_{691}(p)=d$ where $d+1$ is an odd prime are $d=2, 4, 22,$ and $690$.
For the cases where $d=2$ and $4$ respectively, Lemma~\ref{DiophantineCriterion} (1-2) implies that
 $(p,\tau(p))$ would be an integral point on $C^{\pm}_{6,691},$ and that $(p, 2\tau(p)^2-3p^{11})$ would be an integral point on
 $H^{\pm}_{11,691}.$ Lemma~\ref{Hyper11_19} (2-3) and Lemma~\ref{Plus691} show that no such points exist.
By Lemma~\ref{DiophantineCriterion} (3), the remaining cases (i.e. $d=22$ and $690$)  correspond to the Thue equations
$F_{22}(p^{11},\tau(p)^2)=\pm 691$ and 
$F_{690}(p^{11},\tau(p)^2)=\pm 691.$
Lemma~\ref{Monster} (3) and (4) show that there are no such integer solutions.

The arguments above show that $\tau(n)\not \in \{\pm1 , \pm 3, \pm 5, \pm 7, \pm 691\}.$
The remaining cases are special cases of Theorem~\ref{LehmerGeneral} (6) and (9) and are  proved below.
\end{proof}

\begin{proof}[Proof of Theorem~\ref{LehmerGeneral}]  By hypothesis, for primes $p\nmid 2N$ we have that
$a_f(p)$ is even. For such primes, Theorem~\ref{Newforms} (2) implies that
$a_f(p^{m})$ is odd if and only if $m$ is even.
Suppose that $p$ is a prime for which $p\mid a_f(p),$ which includes those primes $p\mid 2N$ by Theorem~\ref{Newforms} (4).  Theorem~\ref{Newforms} (2) and (4) imply that
$p^m\mid a_f(p^m)$. Therefore, we do not need to consider these coefficients in the remainder of the proof.

It suffices to consider the Lucas sequences corresponding to $A=a_f(p)$ and $B=p^{2k-1}$, when $p\nmid a_f(p)$.
By applying Lemma~\ref{Part1DefectivePrimality} (2) (as above in the proof of Theorem~\ref{Lehmer135}), we may assume that
 $\{1, a_f(p), a_f(p^2),\dots\}$ is a Lucas sequence
without any defective terms. To establish this, we must show that
$B_{1,k}^{1,\pm},$ which are the same as $C_{k,3}^{\pm},$  have no suitable integer points.
Since we only consider weights for which $\gcd(3\cdot 5\cdot 7\cdot 11\cdot 13, 2k-1)\neq 1$, it suffices
to show that $C_{d,3}^{\pm}$ has no such points for $d\in \{2, 3, 4, 6, 7\}$. Lemma~\ref{HYPER} confirms this requirement for
these ten curves.

The first claim of the theorem now follows from Proposition~\ref{One}. To prove the remaining claims we apply Theorem~\ref{LehmerVariantGeneral}.
 Namely, if $|a_f(n)|=\ell,$ then $n=p^{d-1}$, where
 $d\mid \ell (\ell^2-1)$ is an odd prime. The existence of such coefficients can be ruled out
with Lemma~\ref{DiophantineCriterion}, which reduces the proof to a case-by-case search for suitable integral points on hyperelliptic curves and solutions to Thue equations which were considered in the previous section.
If $a_f(p^2)=\pm \ell$, then  $(p,a_f(p))\in C_{k,\ell}^{\pm}$.
If $a_f(p^4)=\pm \ell$, then  $(p, 2_f(p)^2-3p^{2k-1})\in
H_{2k-1,\ell}^{\pm}$. 
Obviously, it suffices to study curves $C_{d,\ell}^{\pm}$  (resp. $H_{2d-1,\ell}^{\pm}$) with $d\mid (2k-1)$.
Finally, if $a_f(p^{d-1})=\pm \ell$  with  $d\geq 7,$ then $(p^{2k-1}, a_f(p)^2)$ is a solution to
$F_{d-1}(X,Y)=\pm \ell.$
By Lemmas~\ref{Monster}, \ref{HYPER}, \ref{AnnalsCorollary}, and \ref{Hyper11_19} (i.e. inspecting the tables in the Appendix), 
there are no such integral points (sometimes under GRH) in the cases claimed by the theorem.
\end{proof}

\section{Lehmer's speculation for large weight newforms}

We conclude this paper with the proof of Theorem~\ref{Power}. To prove this result,
we make use of Theorem~\ref{LehmerVariantGeneral}, which in turn reduces the problem
to a search for integer points on suitable curves by Lemma~\ref{DiophantineCriterion}.
Namely, we show, for each $\ell^m$, that the
finitely many Diophantine conditions have no integer solutions when the newform weights are (effectively) sufficiently large.
To derive these conclusions, we employ a deep theorem of Baker and W\"ustholz \cite{BW} in the theory
of linear forms in logarithms, and  work of Tzanakis and de Weger \cite{TW}
on Thue equations.

\subsection{Some Diophantine equations}
Here we prove some Diophantine results concerning families of Lebesgue--Ramanujan--Nagell type equations which are
of independent interest.
To make them precise,
for $\ell\in \{3, 5\}, \varepsilon\in \{\pm\},$ and $m\in \Z^{+}$, we define
\begin{equation}
T^{\varepsilon}(\ell,m):= \begin{cases} 2m+10^{32}\sqrt{m} \ \ \ \ \ &{\text {\rm if $\varepsilon=+$ and $\ell=3$}},\\
2m+10^{23}\sqrt{m} \ \ \ \ \ &{\text {\rm if $\varepsilon=-, m$ odd, and $\ell=3$}},\\
2m+10^{13}\sqrt{m} \ \ \ \ \ &{\text {\rm if $\varepsilon=-, m$ even, and $\ell=3$}},\\

3m+10^{24}\sqrt{m} \ \ \ \ \ &{\text {\rm if $\varepsilon=\pm , m$ odd, and $\ell=5$}},\\
3m+10^{30}\sqrt{m} \ \ \ \ \ &{\text {\rm if $\varepsilon=+, m$ even, and $\ell=5$}},\\
3m+10^{13}\sqrt{m} \ \ \ \ \ &{\text {\rm if $\varepsilon=-, m$ even, and $\ell=5$}}.\\                            \end{cases}
\end{equation}           
Furthermore, we define  $U^{\varepsilon}(m)$ by
                      \begin{equation}
U^{\varepsilon}(m):= \begin{cases}  3m+10 ^{24}\sqrt{m}\ \ \ \ \ &{\text {\rm if $\varepsilon=\pm $ and $m$ odd}},\\
3m+10^{30}\sqrt{m} \ \ \ \ \ &{\text {\rm if $\varepsilon=+$ and $m$ even}},\\
3m+10^{13}\sqrt{m} \ \ \ \ \ &{\text {\rm if $\varepsilon=-$ and $m$ even}}.
                                   \end{cases}
\end{equation}

\begin{thm}\label{Explicit35}
If $\ell \in \{3, 5\},$ $\varepsilon\in \{\pm \}$, and $m\in \Z^{+}$, then the following are true.
\newline
(1)  If $n>T^{\varepsilon}(\ell,m)=O_{\ell}(m),$ then there are no integer points\footnote{We switch $X$ and $Y$ here to be consistent with the literature on Lebesgue--Ramanujan--Nagell equations.} $(X,Y),$ with $Y\not \in \{0,\pm 1\}$, on
\begin{align}\label{ell3}
X^2 +\varepsilon \ell^m=Y^n.
\end{align}
\newline
\noindent
(2) If $n>U^{\varepsilon}(m)=O_{\ell}(m),$ then there are no integer points $(X,Y),$ with $Y\neq 0$, on
\begin{align}\label{ell5}
X^2+\varepsilon 4\cdot 5^m = Y^n.
\end{align}
\end{thm}

\subsection{A theorem of Baker and W\"ustholz}
To prove Theorem~\ref{Explicit35}, we make use of the following classical result
of Baker and W\"ustholz \cite{BW} on linear forms in logarithms.


\begin{thm}[p. 20 of \cite{BW}]\label{BW}
Let $\alpha_1,\ldots,\alpha_r$ be algebraic numbers and $b_1,\ldots,b_r$ be rational integers. If $\Lambda:=b_1\log\alpha_1+\cdots+b_r\log\alpha_r$ (note. where the logarithms have their principal values such that $-\pi<\mathrm{Im}(\log \alpha)\leq\pi$) is nonzero,
 then we have
\begin{align*}
    \log|\Lambda|>-C(r,d)\log(\mathrm{max}\left\{e,B\right\})\prod_{i=1}^{r}h'(\alpha_i),
\end{align*}
where $d:=[\mathbb{Q}(\alpha_1,\ldots,\alpha_r):\mathbb{Q}]$, $B:=\mathrm{max}\left\{|b_1|,\ldots,|b_r|\right\}$,
\begin{align*}
    C(r,d):=18(r+1)!~r^{r+1}(32d)^{r+2}\log(2rd),
\end{align*}
and $h'(\alpha):=\mathrm{max}\left\{h(\alpha)/d,|\log\alpha|/d,1/d\right\}$, where $h(\alpha)$ is the logarithmic Weil height of $\alpha$.
\end{thm}

This deep theorem can be applied to the  Diophantine equations in (\ref{ell3}) and (\ref{ell5}).
We shall now assume that $n$ is fixed for the remainder of this discussion.
Namely, we view potential integer points as factorizations, in the ring of integers
of  the quadratic fields $K=\Q(\sqrt{-\varepsilon \ell^m}),$
given by
\begin{align*}
(X+\sqrt{-\varepsilon \ell^m})(X-\sqrt{-\varepsilon \ell^m}) =Y^n \ \  \ {\text {\rm and}}\ \  \
(X+2\sqrt{-\varepsilon \ell^m})(X-2\sqrt{-\varepsilon \ell^m}) =Y^n.
\end{align*}
Namely, if $[K:\Q]=2$ and $h_K=1$, then we have $\beta\in\mathcal{O}_K$ such that $N_{K/\mathbb{Q}}(\beta)=Y$ and 
\begin{align*}
    (X+\sqrt{-\varepsilon \ell^m})=\beta^n ~(\mathrm{mod}~ \mathcal{O}_{K}^{\times})
    \ \ \ {\text {\rm and}}\ \ \ 
    (X+2\sqrt{-\varepsilon \ell^m})=\beta^n ~(\mathrm{mod}~ \mathcal{O}_{K}^{\times}).
\end{align*}
If $K$ does not have class number one, then we may pick $\beta\in\mathcal{O}_K$ such that $  N_{K/\mathbb{Q}}(\beta)=Y^{h_K}$ and consider $\beta^{{n}/{h_K}}$ instead.
This only applies when $\varepsilon=1, \ell=5$ and $m$ is odd, in which case  $h_{\Q(\sqrt{-5})}=2$. 
In these cases we let $\overline{\beta}$ denote the Galois conjugate of $\beta$. Finally, if $K=\Q$, then we may pick $\beta,\overline{\beta}\in\Z$  (abusing notation) such that $\beta\overline{\beta}=Y$ and $|{\beta}|\leq \sqrt{|Y|}.$
In each case, the algebraic integer $\beta$ is uniquely determined up to unit.

Given such a $\beta$, we construct a corresponding linear form in logarithms arising from $\beta/\overline{\beta}.$ 
For convenience, we denote the relevant fundamental units by
$w_{3}:=2+\sqrt{3}$ and $w_{5}:=1/2+\sqrt{5}/2$, and we denote the 6th root of unity by $w_{-3}:=1/2+\sqrt{-3}/2.$
By taking logarithms,  we obtain a triple of integers $0\leq j_4\leq 3, 0\leq j_6\leq 5,$ and $0\leq j_n<n-1,$ for which one of the
corresponding forms (depending on $\varepsilon, \ell$ and the parity of $m$), say
$\Lambda_{T^{\varepsilon}(\ell,m)}$ and  $\Lambda_{U^{\varepsilon}(m)},$ is given by
\begin{equation}
\Lambda_{T^{\varepsilon}(\ell,m)}:= \begin{cases} j_6\log({\overline{w}_{-3}}/ w_{-3})-n\log({\overline{\beta}}/\beta)+ki\pi \ \ \ \ \ &{\text {\rm if $\varepsilon=+, m$ odd, and $\ell=3$}},\\
j_4\log({\overline{i}}/ i)-n\log({\overline{\beta}}/\beta)+ki\pi \ \ \ \ \ &{\text {\rm if $\varepsilon=+, m$ even, and $\ell=3$}},\\
-(n/2)\log({\overline{\beta}}/\beta)+ki\pi \ \ \ \ \ &{\text {\rm if $\varepsilon=+, m$ odd, and $\ell=5$}},\\
j_4\log({\overline{i}}/ i)-n\log({\overline{\beta}}/\beta)+ki\pi \ \ \ \ \ &{\text {\rm if $\varepsilon=+, m$ even, and $\ell=5$}},\\
j_n\log(\overline{w}_{3}/w_{3})-n\log({\overline{\beta}}/\beta) \ \ \ \ \ &{\text {\rm if $\varepsilon=-, m$ odd, and $\ell=3$}},\\
-n\log(\overline{\beta}/\beta) \ \ \ \ \ &{\text {\rm if $\varepsilon=-, m$ even, and $\ell=3$}},\\
j_n\log(\overline{w}_{5}/w_{5})-n\log({\overline{\beta}}/\beta)  \ \ \ \ \ &{\text {\rm if $\varepsilon=-, m$ odd, and $\ell=5$}},\\
-n\log(\overline{\beta}/\beta) \ \ \ \ \ &{\text {\rm if $\varepsilon=-, m$ even, and $\ell=5$}},\\
\end{cases}
                    \end{equation}     
and
                      \begin{equation}
\Lambda_{U^{\varepsilon}(m)}:= \begin{cases}  -(n/2)\log({\overline{\beta}}/\beta)+ki\pi \ \ \ \ \ &{\text {\rm if $\varepsilon=+$ and $m$ odd}},\\
j_4\log({\overline{i}}/ i)-n\log({\overline{\beta}}/\beta)+ki\pi \ \ \ \ \ &{\text {\rm if $\varepsilon=+$ and $m$ even}},\\
j_n\log(\overline{w}_{5}/w_{5})-n\log({\overline{\beta}}/\beta) \ \ \ \ \ &{\text {\rm if $\varepsilon=-$ and $m$ odd}},\\
-n\log(\overline{\beta}/\beta) \ \ \ \ \ &{\text {\rm if $\varepsilon=-$ and $m$ even}},
                 \end{cases}
\end{equation}          
where $k\in\mathbb{Z}$ with $
|\Lambda_{T^{+}(\ell,m)}|,~|\Lambda_{U^{+}(m)}|<\pi$. 
The next lemma bounds these quantities.

\begin{lem}\label{lambdabound} Assuming the notation and hypotheses above, the following are true.

\noindent
(1)  If $n>2\log(4\sqrt{\ell^m})/\log |Y|$ and $(X,Y)$ is an integer point on (\ref{ell3}),  with $Y\not \in \{0,\pm 1\}$, then
\begin{align*}
    |\Lambda_{T^{\varepsilon}(\ell,m)}|\leq 2.78\cdot \frac{\sqrt{\ell^{m}}}{|Y|^{\frac{n}{2}}}.
\end{align*}

\noindent
(2) If $n>2\log(8\sqrt{5^m})/\log |Y|$, and $(X,Y)$ is an integer point on (\ref{ell5}), with $Y\neq 0$, then 
\begin{align*}
    |\Lambda_{U^{\varepsilon}(m)}|\leq 5.56\cdot \frac{\sqrt{5^{m}}}{|Y|^{\frac{n}{2}}}.
\end{align*}
\end{lem}
\begin{proof}
By the definition of $\Lambda_{T^{\varepsilon}(\ell,m)}$, we directly find that
\begin{align}\label{lambda1}
    |e^{\Lambda_{T^{\varepsilon}(\ell,m)}}-1|=
    \left|\frac{X+\sqrt{\pm\ell^m}}{X-\sqrt{\pm\ell^m}}-1\right|\leq\frac{2\sqrt{\ell^m}}{{|Y|^{\frac{n}{2}}}}.
\end{align}
For $|z|<1/2$, we note that
$|\log(1+z)|\leq 1.39\cdot |z|.$
Also, we note that the hypothesis on $n$ gives $|e^{\Lambda_{T^{\varepsilon}(\ell,m)}}-1|<1/2$.
Hence, we obtain (1), the claimed inequality
\begin{align*}
    |\Lambda_{T^{\varepsilon}(\ell,m)}|\leq 1.39\cdot |e^{\Lambda_{T^{\varepsilon}(\ell,m)}}-1|=2.78\cdot \frac{\sqrt{\ell^m}}{{|Y|^{\frac{n}{2}}}}.
\end{align*}
 The same method gives (2), after noting that $Y=\pm 1$  has no integer point on (\ref{ell5}). 
\end{proof}

\subsection{Proof of Theorem~\ref{Explicit35}}
For brevity, we only consider  when $\ell=3$ and $\varepsilon=-$,  as  the same method applies to all of the cases.
Suppose that there is an integer point $(X,Y)$ on $X^2+3^m=Y^n$.
Therefore, there is an integer $0\leq j_6\leq 5$ and an algebraic integer $\beta \in \Q(\sqrt{-3})$ for which
$N_{K/\Q}(\beta)=Y$ and
\begin{align*}
    (X+\sqrt{-3^m})=\frac{\beta^n}{w_{-3}^{j_6}}.
\end{align*}
In particular, if $m$ is odd, then we have
\begin{align*}
    \Lambda_{T^{\varepsilon}(\ell,m)}=j_6\log({\overline{w}_{-3}}/ w_{-3})-n\log({\overline{\beta}}/\beta)+ki\pi
    =j_6\log({\overline{w}_{-3}}/ w_{-3})-n\log({\overline{\beta}}/\beta)+k\log(-1).
\end{align*}
Since $\Lambda_{T^{\varepsilon}(\ell,m)}\neq 0,$ Theorem \ref{BW} implies that
\begin{align*}
    \log|\Lambda_{T^{\varepsilon}(\ell,m)}|>-C(3,2)h'(\overline{w}_{-3}/w_{-3})h'(\overline{\beta}/\beta)h'(-1)\log(\mathrm{max}\left\{e,j_6,n,|k|\right\}.
\end{align*}
Furthermore, by a short calculation, we get
\begin{displaymath}
\begin{split}
&h'(\overline{w}_{-3}/w_{-3})\leq\frac{\pi}{3},\\
 &h'(\overline{\beta}/\beta)\leq\mathrm{max}\left\{\log |Y|, \pi\right\}\\
 &h'(-1)\leq\frac{\pi}{2},~\mathrm{max}\left\{e,j_6,n,|k|\right\})\leq n+5.
\end{split}
\end{displaymath}
Therefore, Theorem~\ref{BW} implies that
\begin{align*}
     \log|\Lambda_{T^{\varepsilon}(\ell,m)}|&>-\frac{\pi^2}{6}C(3,2)\mathrm{max}\left\{\log |Y|, \pi\right\}\log(n+5).
\end{align*}
However, Lemma \ref{lambdabound} (1) gives
\begin{align*}
 \log(2.78\cdot\sqrt{3^m})-\frac{n}{2}\cdot {\color{black}\log |Y|} >\log|\Lambda_{T^{\varepsilon}(\ell,m)}|&>-\frac{\pi^3}{6}C(3,2)\log(n+5)\cdot {\color{black}\log |Y|},
\end{align*}
which in turn implies that
$$
    \log(2.78\cdot\sqrt{3^m})-\frac{n}{2}\log 2  >-\frac{\pi^3}{6}C(3,2)\sqrt{n+4}.
$$
Since we have $C(3,2)=18(4)!~3^4(6
4)^5\log(12),$ a direct calculation shows that we must have
\begin{align*}
    n\leq 1.6m+(60\sqrt{m}+5.9)\cdot 10^{30},
  \end{align*}
which gives a constant that is smaller than the claimed $M^{-}(3,m).$ Taking into account even $m$,
a similar calculation gives $n< 1.6m +(9.4\sqrt{m}+1.4)\cdot 10^{31}.$ The claimed $M^{-}(3,m)$ is
a ``rounded up'' version of the maximum of these two constants.

\subsection{Proof of Theorem~\ref{Power}} Suppose that $\ell^m$ is a power of an odd prime. Thanks to Theorem~\ref{LehmerVariantGeneral}, if $a_f(n)=\pm \ell^m,$
then $n=p^{d-1}$, where $p$ and $d\mid \ell(\ell^2-1)$ are odd primes. 
For each $d$, Lemma~\ref{DiophantineCriterion} gives an integer
point on an elliptic or hyperelliptic curve, or gives an integer solution to a Thue equation.

If $\ell =3$ (resp. $\ell=5$), then we find that the only possibility is $d=3$ (resp. $d=3, 5$).  This leads to the equations 
in Theorem~\ref{Explicit35}, which in turn gives the claimed bounds in these cases.
Turning to $\ell \geq 7$, we note for $d=3$ (resp. $5$) that one
can argue again as in the proof of Theorem~\ref{Explicit35} to conclude that $a_f(p^2)\neq \pm \ell^m$ 
(resp. $a_f(p^4)\neq \pm \ell^m$) for $f$ with (effectively) sufficiently large  weight $2k$.
For any $d\geq 7$, Lemma~\ref{DiophantineCriterion} (3)  gives the
integer solution $(X,Y)=(p^{2k-1}, a_f(p^2))$ to the Thue equation
$$
F_{d-1}(X,Y)=\pm \ell^m.
$$
As an implementation of Baker's theory of linear forms in logarithms, 
a well-known paper of Tzanakis and de Weger (see p. 103 of \cite{TW})  on Thue equations gives a method for effectively determining an
upper bound\footnote{The reader should switch the roles of $X$ and $Y$ when applying the discussion in \cite{TW}.}  for $|X|$ of any integer point satisfying $F_{d-1}(X,Y)=\pm \ell^m$, which in turn leads to an upper bound for the weight $2k$.  The linearity of these constants in $m$ aspect follows from the formal
taking of a logarithm in these Diophantine equations.

\newpage

\section{Appendix}

\begingroup
\setlength{\tabcolsep}{3pt} 
\renewcommand{\arraystretch}{1.5}
\begin{center} 
\begin{table}[!ht]
\begin{tabular}{|c|c|}
\multicolumn{2}{c}{} \\ \hline
$(A,B)$ & Defective $u_n(\alpha, \beta)$ \\ \hline \hline
\multirow{2}{3.5em}{$(\pm 1,2^1)$} & $u_5 = -1$, $u_7 = 7$, $u_8 = \mp 3$, $u_{12} = \pm 45$, \\ & $u_{13} = -1$, $u_{18} = \pm 85$, $u_{30} = \mp 24475$ \\ \hline
$(\pm 1,3^1)$ & $u_5 = 1$, $u_{12} = \pm 160$ \\ \hline
$(\pm 1, 5^1)$ & $u_7 = 1$, $u_{12} = \mp 3024$ \\ \hline
$(\pm 2, 3^1)$ & $u_3 = 1$, $u_{10} = \mp 22$ \\ \hline
$(\pm 2,7^1)$ & $u_8 = \mp 40$ \\ \hline
$(\pm 2, 11^1)$ & $u_5 = 5$ \\ \hline
$(\pm 4, 5^1)$ & $u_6=\pm 44$\\ \hline
$(\pm 5, 7^1)$ & $u_{10} = \mp 3725$ \\ \hline
$(\pm 3, 2^3)$ & $u_3 = 1$ \\ \hline
$(\pm 5, 2^3)$ & $u_6 = \pm 85$ \\ \hline
\end{tabular}
\medskip
\captionof{table}{\textit{Sporadic examples of defective $u_n(\alpha, \beta)$ satisfying (\ref{Modularity})}} 
\label{table1}
\end{table}
\end{center}
\endgroup

\noindent
The families of defective Lucas numbers satisfying (\ref{Modularity}) are given
 by the following curves.
\begin{equation}\label{Table2Curves}
\begin{split}
 \ \ \ \ \ \ &B_{1, k}^{r, \pm} : Y^2 = X^{2k-1} \pm 3^r, \ \ \ \   B_{2,k} : Y^2 = 2X^{2k-1} - 1, \ \ \ B_{3, k}^\pm : Y^2 = 2X^{2k-1} \pm 2, \\ \ \ B_{4,k}^r &: Y^2 =  3X^{2k-1} + (-2)^{r+2}, \ \ \  B_{5,k}^{\pm} : Y^2=3X^{2k-1} \pm 3, \ \ \  B_{6,k}^{r,\pm} : Y^2 = 3X^{2k-1} \pm 3 \cdot 2^r.
\end{split}
\end{equation}

\begingroup
\setlength{\tabcolsep}{3pt} 
\renewcommand{\arraystretch}{2.5}
\begin{center} 
\begin{table}[!ht]
\begin{small}
\begin{tabular}{|c|c|c|}
 \hline
$(A,B)$ & Defective $u_n(\alpha, \beta)$ & Constraints on parameters \\ \hline \hline
$(\pm m, p)$ & $u_3 = -1$ & $m>1$ and $p = m^2+1$ \\ \hline
\multirow{2}{*}{}
$(\pm m, p^{2k-1})$ &
$u_3 = \varepsilon 3^r$ &
$\begin{aligned} &\ \ \ \ \ \ (p, \pm m)\in B_{1, k}^{r, \varepsilon} \text{ with } 3\nmid m,\\ 
&(\varepsilon,r,m)\neq (1,1,2),
\text{ and } m^2 \geq 4\varepsilon 3^{r-1} \end{aligned}$\\ \hline
{\color{black}$(\pm m, p^{2k-1})$} & {\color{black}$u_4 = \mp m$} & {\color{black}$(p,\pm m) \in B_{2,k}$ with $m > 1$ odd} \\ \hline
$(\pm m, p^{2k-1})$ & $u_4 = \pm 2{\color{black}\varepsilon}m$ & $\begin{aligned}(p,\pm m)\in B_{3,k}^\varepsilon &\text{ with } {\color{black}(\varepsilon, m) \not = (1,2)}\\ &\text{  and } m > 2
\text{  even}
\end{aligned}$ \\ \hline
$(\pm m, p^{2k-1})$ & 
$u_6 = {\color{black}\pm (-2)^rm(2m^2+(-2)^r)/3}$ &
$\begin{aligned}(p, &\pm m)\in B_{4,k}^r \text{ with } \gcd(m,6) = 1, \\
&{\color{black}(r,m) \not = (1,1)}, \text{ and } {\color{black}m^2 \geq (-2)^{r+2}}\end{aligned}$ \\ \hline
$(\pm m, p^{2k-1})$ & ${\color{black}u_6=\pm \varepsilon m(2m^2+3\varepsilon)}$ & $(p,\pm m)\in B_{5,k}^\varepsilon$ with $3\mid m$ and $m>3$\\ \hline
$(\pm m, p^{2k-1})$ & $u_6 = \pm 2^{r+1}{\color{black}\varepsilon} m(m^2 + 3 {\color{black}\varepsilon} \cdot 2^{r-1}) $ &$\begin{aligned} (p, \pm m)\in B_{6,k}^{r,\varepsilon}
 \text{ with } m \equiv 3 \bmod{6} \\  \text{and } m^2 \geq 3 {\color{black}\varepsilon} \cdot 2^{r+2}\end{aligned}$ \\ \hline
\end{tabular}
\end{small}
\medskip
\captionof{table}{Parameterized families of defective $u_n(\alpha, \beta)$ satisfying (\ref{Modularity})
\label{table2}
\textit{\newline Notation: $m, k, r\in \Z^{+}$, $\varepsilon = \pm 1$, $p$ is a prime number.}}
\end{table}
\end{center}
\endgroup

\vskip.65in

\begingroup
\setlength{\tabcolsep}{5pt}
\renewcommand{\arraystretch}{2.5}
\begin{center}
\begin{table}[!ht]
\begin{tabular}{|c|l|} \hline
$(a_f(p),p^{2k-1})$ & \multicolumn{1}{c|}{$\widehat{\sigma}(p,m)$} \\[1.5ex] \hline \hline

\multirow{2}{*}{}
$(\pm 3, 2^3)$ & $\begin{aligned} 
\sigma_0(m+1) - 2 & \ \text{ when } 3|(m+1), \\
\sigma_0(m+1) - 1 & \ \text{ otherwise.} \\
\end{aligned}$ \\[1.5ex] \hline

\multirow{2}{*}{}
$(\pm 5, 2^3)$ & $\begin{aligned}
\sigma_0(m+1) - 2 & \ \text{ if } 6|(m+1), \\
\sigma_0(m+1) - 1 & \ \text{ otherwise}.
\end{aligned}$ \\[1.5ex] \hline

\multirow{2}{*}{}
$(\pm m, p^{2k-1})$ & $\begin{aligned}
\sigma_0(m+1) - 4 & \ \text{ if $(p, \pm m) \in S$,} \\
\sigma_0(m+1) - 1 & \ \text{ otherwise.}
\end{aligned}$ \\[1.5ex] \hline
\end{tabular}
\medskip
\captionof{table}{\textit{Lower bounds on $\Omega(a_f(p^m))$ in defective cases for weights $2k \geq 4$.}
\label{table3}
\textit{\newline Notation: $S$ is the collection of all points on any of $B_{1,k}^{r,\pm}, B_{2,k}, B_{3,k}^r, B_{4,k}, B_{5,k}^r$.}}
\end{table}
\end{center}
\endgroup


\bigskip

\begingroup
\setlength{\tabcolsep}{5pt} 
\renewcommand{\arraystretch}{1.9}
\begin{center} 
\begin{table}[!ht]
\begin{tabular}{|c|c|}
\multicolumn{2}{c}{} \\ \hline
$(d, D)$ & Integer Solutions to $F_{d-1}(X,Y)=D$  \\ \hline \hline
$(7, \pm 7)$ & $(\pm 1, \pm 4), ( \pm 2, \pm 1), (\mp 3, \mp 5)$\\ \hline
\multirow{2}{3.5em}{$(7,  \pm 13)$} & $( \pm 3,  \pm 10), (\pm 2,  \pm 7), ( \pm 3, \pm 4), (\pm 4, \pm 1),$ \\ & $( \pm 3,  \pm 1), (\mp 1, \pm 1), (\mp 2, \mp 5), (\mp 5,\mp 8), (\mp 7, \mp 11)$ \\ \hline
$(7, \pm 29)$ & $\begin{aligned}(\mp 6, \mp 1), (\mp 5, &\mp 16), (\mp 4, \mp 7), (\pm 1, \pm 5), \\&(\pm 3, \pm 2), (\pm 11, \pm 17)\end{aligned}$\\ \hline
$(11, \pm 11), (19, \pm 19),$  &\multirow{2}{3.5em}{$( \pm 1,\pm 4)$}\\ 
$(23, \pm 23), (31, \pm 31)$ &\\ \hline
$(11, \pm 23)$ & $( \pm 3, \pm 2), ( \pm 2, \pm 1), (\mp 2,\mp 3)$\\ \hline
$(13, 13), (17, 17), (29, 29), (37,37)$ & $(-1, -4), (1,4)$\\ \hline
$(13, -13),  (17, -17),$  & \multirow{2}{3.5em}{$\varnothing$}\\ 
$(29, -29), (37,-37)$ &\\ \hline
$(19, \pm 37)$ & $(\mp 2, \mp 5)$ \\ \hline
\end{tabular}
\smallskip
\captionof{table}{\textit{Solutions for the Thue equations where $D=\pm \ell$ and $7\leq \ell \leq 37$} }
\label{thuetable}
\end{table}
\end{center}
\endgroup


\begingroup
\setlength{\tabcolsep}{5pt} 
\renewcommand{\arraystretch}{1.9}
\begin{center} 
\begin{table}[!ht]
\begin{tabular}{|c|c|}
\multicolumn{2}{c}{} \\ \hline
$(d, D)$ & Integer Solutions to $F_{d-1}(X,Y)=D$  \\ \hline \hline
$(7, \pm 41) $ & $(\mp 3, \mp 7), (\mp 1, \pm 2), (\pm 4, \pm 5)  $\\ \hline
$\begin{aligned}&(41, 41), (53, 53), (61,61), \\
&(73, 73), (89,89), (97, 97)\end{aligned}$ & $(-1, -4), (1,4)$ \\ \hline
$(41, -41), (23, \pm 47), (13, 53), (53,-53), (29, \pm 59),$ &  \multirow{3}{3.5em}{$\varnothing$} \\ 
$(31, \pm 61), (61, -61), (17,-67), (37, \pm 73), (73,-73),$ & \\
$(13,-79), (41, \pm 83), (89,-89), (97,-97)$ & \\   \hline
$(7,  \pm 43)$ & $(\mp 3, \mp 8), (\mp 2, \pm 1), (\pm 5, \pm 7) $\\ \hline
$(11, \pm 43)$ & $(\mp 3, \mp 5), (\pm 2, \pm 5)$\\ \hline
$(43, \pm 43), (47, \pm 47), (59, \pm 59), (67,\pm 67),$ &  \multirow{2}{3.5em}{$( \pm 1,\pm 4)$}\\
$(71, \pm 71), (79, \pm 79), (83, \pm 83)$ &  \\ \hline
$(13, -53), (17,67)$ & $ (-2,-3), (2, 3)$\\ \hline
$(11, \pm 67)$ & $(\mp 7, \mp 12), (\mp 3, \mp 11), (\mp 2, \mp 7)$ \\ \hline

\multirow{2}{3.5em}{$(7,\pm 71)$} & $(\mp 16, \mp 25), (\mp 5, \mp 9), (\pm 1, \pm 6),$\\ & $(\pm 4, \pm 3), (\pm 7, \pm 23), (\pm 9, \pm 2)$\\ \hline
$(13, 79)$ & $(-2,-5),(2,5)$\\ \hline

\multirow{2}{3.5em}{$(7,\pm 83)$} & $(\mp 8, \mp 13), (\mp 7, \mp 1), (\mp 6, \mp 19),$\\
&  $(\pm 3, \pm 11), (\pm 5, \pm 2), (\pm 13, \pm 20)$\\ \hline
$(11, \pm 89)$ & $(\mp 1, \pm 1) $\\ \hline
$(7,\pm 97)$ & $(\mp 4, \mp 11), (\mp 3, \pm 1), (\pm 7, \pm 10)$\\ \hline

\end{tabular}
\medskip
\captionof{table}{\textit{Solutions (with GRH) to the Thue equations where $D=\pm \ell$ and $41\leq \ell \leq 97$} }
\label{thueGRHtable}
\end{table}
\end{center}
\endgroup

\bigskip

\begingroup
\setlength{\tabcolsep}{3pt} 
\renewcommand{\arraystretch}{1.9}
\begin{center}
\begin{table}
\begin{tabular}{|c|c|c|c|c|c|cl}
 \hline
$\ell$ & {\text {\rm $C_{2,\ell}^{+}$}} & {\text {\rm $C_{3,\ell}^{+}$}} & {\text {\rm $C_{4,\ell}^{+}$}}  & {\text {\rm $C_{6,\ell}^{+}$}} & {\text {\rm $C_{7,\ell}^{+}$}}\\  \hline \hline
$3$ & $(1,\pm 2)$ & $(1,\pm 2)$ & $(1,\pm 2)$  & $(1, \pm 2)$ & $(1,\pm 2)$  \\ \hline
$5$ & $(-1,\pm 2)$ & $(-1, \pm 2)$ & $(-1,\pm 2)$  &$(-1, \pm 2)$ & $(-1.\pm 2)$ \\ \hline
$\begin{aligned}7, &23, 29, 47,
 53,\\  &59, 61, 67, 83
 \end{aligned} $ & $\varnothing$ & $\varnothing$ & $\varnothing$  & $\varnothing$ & $\varnothing$ \\ \hline
$11$ & $\varnothing$ & $(5,\pm 56)$ & $\varnothing$  & $\varnothing$ & $\varnothing$ \\ \hline
$13$ & $\varnothing$ & $(3,\pm 16)$ & $\varnothing$  & $\varnothing$  & $\varnothing$ \\ \hline
$17$ & $\begin{aligned} &(-2,\pm 3), (-1,\pm 4), (2,\pm 5), &\\ &(4,\pm 9), (8, \pm 23) 
(43, \pm 282), \\ &(52, \pm 375), (5234, \pm 378661)\end{aligned}$ & $(-1,\pm 4)$ & $(-1,\pm 4)$  & $(-1,\pm 4)$ & $(-1,\pm 4)$ \\ \hline
$19$ & $(5, \pm 12)$ & $\varnothing$ & $\varnothing$  & $\varnothing$ & $\varnothing$ \\ \hline
$31$ & $(-3,\pm 2)$  & $\varnothing$ & $\varnothing$ & $\varnothing$ & $\varnothing$\\ \hline
$37$ & $\begin{aligned} (-1,\pm 6), (3, \pm 8),\\ (243,\pm 3788) \end{aligned}$  & $(-1,\pm 6), (27, \pm 3788)$ & $(-1,\pm 6)$ & $(-1,\pm 6)$ & $(-1,\pm 6)$\\ \hline
$41$ & $(2,\pm 7)$   & $(-2,\pm 3)$ & $(2,\pm 13)$ & $\varnothing$ & $\varnothing$\\ \hline
$43$ & $(-3,\pm 4) $  & $\varnothing$ & $\varnothing$ & $\varnothing$ & $\varnothing$\\ \hline
$71$ & $(5, \pm 14)$  & $\varnothing$ & $\varnothing$ & $\varnothing$ & $\varnothing$\\ \hline
$73$ & $\begin{aligned}(-4,\pm 3), (2,\pm 9), \\ (3,\pm 10),
 (6,\pm 17),\\  (72,\pm 611), (356, \pm 6717)\end{aligned}$  & $\varnothing$ & $\varnothing$ & $\varnothing$ & $\varnothing$\\ \hline
$79$ & $(45,\pm 302) $  & $\varnothing$ & $\varnothing$ & $\varnothing$ & $\varnothing$\\ \hline
$89$ & $\begin{aligned}(-4, \pm 5), (-2, \pm 9), \\
(10,\pm 33), (55,\pm 408)\end{aligned}$  & $(2,\pm 11)$ & $\varnothing$ & $\varnothing$ & $\varnothing$\\ \hline
$97$ & $\varnothing $  & $\varnothing$ & $(2,\pm 15)$ & $\varnothing$ & $\varnothing$\\ \hline
\end{tabular}
\medskip
\captionof{table}{\textit{Integer points on $C_{d,\ell}^{+}$}}
\label{Cplustable}
\end{table}
\end{center}
\endgroup

\begingroup
\setlength{\tabcolsep}{3pt} 
\renewcommand{\arraystretch}{1.9}
\begin{center}
\begin{table}
\begin{tabular}{|c|c|c|c|c|c|}
 \hline
$\ell$ & {\text {\rm $C_{2,\ell}^{-}$}} & {\text {\rm $C_{3,\ell}^{-}$}} & {\text {\rm $C_{4,\ell}^{-}$}}  & {\text {\rm $C_{6,\ell}^{-}$}} & {\text {\rm $C_{7,\ell}^{-}$}}\\  \hline \hline
$\begin{aligned} &\ \ 3, 5, 17, 29, 37,\\  &41, 43,  59, 73, 97
\end{aligned} $&  $\varnothing$ & $\varnothing$ & $\varnothing$  & $\varnothing$  & $\varnothing$\\ \hline
$7$ & $(2,\pm 1), \ (32,\pm 181)$ & $(2,\pm 5), \ (8,\pm 181)$ & $(2, \pm 11)$ &  $\varnothing$ & $\varnothing$ \\ \hline
$11$ & $(3,\pm 4), \ (15, \pm 58)$ & $\varnothing$ & $\varnothing$  & $\varnothing$ & $\varnothing$ \\ \hline
$13$ & $(17, \pm 70)$ & $\varnothing$ & $\varnothing$  & $\varnothing$ & $\varnothing$ \\ \hline
$19$ & $(7, \pm 18)$ & $(55, \pm 22434)$ & $\varnothing$  & $\varnothing$  & $\varnothing$ \\ \hline
$23$& $(3,\pm 2)$ & $(2,\pm 3)$ & $\varnothing$  & $(2,\pm 45)$ & $\varnothing$\\ \hline
$31$ & $\varnothing$  & $(2,\pm 1)$  & $\varnothing$ & $\varnothing$ & $\varnothing$ \\ \hline
$47$ & $(6, \pm 13), (12, \pm 41), (63, \pm 500) $  & $(3, \pm 14)$ & $(2,\pm 9)$ & $\varnothing$ & $\varnothing$\\ \hline
$53$ & $(9,\pm 26), (29, \pm 156)$  & $\varnothing$ & $\varnothing$ & $\varnothing$ & $\varnothing$\\ \hline
$61$ & $(5,\pm 8)$   & $\varnothing$ & $\varnothing$ & $\varnothing$ & $\varnothing$\\ \hline
$67$ & $(23, \pm 110)$  & $\varnothing$ & $\varnothing$ & $\varnothing$ & $\varnothing$\\ \hline
$71$ & $(8,\pm 21)$  & $\varnothing$ & $(3,\pm 46)$ & $\varnothing$ & $\varnothing$\\ \hline
$79$ & $(20, \pm 89)$  & $\varnothing$ & $(2,\pm 7)$ & $\varnothing$ & $\varnothing$\\ \hline
$83$ & $(27, \pm 140)$  & $\varnothing$ & $\varnothing$ & $\varnothing$ & $\varnothing$\\ \hline
$89$ & $(5,\pm 6) $  & $\varnothing$ & $\varnothing$ & $\varnothing$ & $\varnothing$\\ \hline
\end{tabular}
\medskip
\captionof{table}{\textit{Integer points on $C_{d, \ell}^{-}$}}
\label{Cminustable}
\end{table}
\end{center}
\endgroup

\bigskip

\begingroup
\setlength{\tabcolsep}{5pt} 
\renewcommand{\arraystretch}{1.6}
\begin{center}
\begin{table}
\begin{tabular}{|c|c|c|c|c|c|c|c|c|c|}
 \hline
$\ell$ & {\text {\rm $H_{3,\ell}^{-}$}} & {\text {\rm $H_{3,\ell}^{+}$}} & {\text {\rm $H_{5,\ell}^{-}$}}  & {\text {\rm $H_{5,\ell}^{+}$}}  & {\text {\rm $H_{7,\ell}^{-}$}}  & {\text {\rm $H_{7,\ell}^{+}$}}    & {\text {\rm $H_{11,\ell}^{-}$}}   & {\text {\rm $H_{13,\ell}^{-}$}}  \\  \hline \hline
$11$ & $\varnothing$ & $(1,7), (7, 767)$ & $\varnothing$  & $(1,7)$ & $\varnothing$ & $(1,7)$ & $\varnothing_*$   &$\varnothing$  \\ \hline
$19$ & $\varnothing$ & $(1,9), (3,61)$ & $\varnothing$  & $(1,9)$ & $\varnothing$ & $(1,9)$ & $\varnothing$ &$\varnothing_*$   \\ \hline
$29$ & $\varnothing$ & $(1,11)$ & $\varnothing$  & $(1,11)$ & $\varnothing$ & $(1,11)$ & $\varnothing_*$ & $\varnothing_*$   \\ \hline
$31$ & $(2,14)$ & $\varnothing$ & $\varnothing$  & $\varnothing$ & $(2,286)$ & $\varnothing$ & $\varnothing_*$   &$\varnothing_*$  \\ \hline
$41$ & $(3,59)$ & $(1,13), (2,22)$ & $\varnothing$  & $(1,13)$ & $\varnothing$ & $(1,13)_*$ & $\varnothing_*$  & $\varnothing_*$  \\ \hline
$59$ & $\varnothing$ & $\varnothing$ & $\varnothing$ &$\varnothing_*$ & $\varnothing$ & $\varnothing_*$ & $\varnothing_*$  &$\varnothing_*$  \\ \hline
$61$ & $\varnothing$ & $\varnothing$ & $\varnothing$  & $\varnothing$ & $\varnothing$ & $\varnothing_*$ & $\varnothing_*$  &$\varnothing_*$  \\ \hline
$71$ & $(2,6), (5, 279)$ & $(1,17)$ & $\varnothing$ & $(1,17)$ & $\varnothing$ & ${\text {\rm {\bf ?}}}$ & $\varnothing_*$   &$\varnothing_*$  \\ \hline
$79$ & $(2,2), (4,142)$ & $\varnothing$ & $\varnothing$  & $\varnothing$ & $\varnothing$ & $\varnothing_*$ & $\varnothing_*$   &$\varnothing_*$  \\ \hline
$89$ & $\varnothing$ & $(1,19), (2,26)$ & $\varnothing$  & $(1,19)_*$, $(2,74)_*$ & $\varnothing$ & $(1,19)_*$ & $\varnothing_*$  & ${\text {\rm {\bf ?}}}$  \\ \hline
\end{tabular}
\medskip
\captionof{table}{\textit{$(|X|, |Y|)$ for integer points on $H_{d,\ell}^{\pm}$ with $\leg{\ell}{5}=1$. \newline (note. GRH assumption indicated by $_*$.)}}
\label{Htable}
\end{table}
\end{center}
\endgroup

\end{document}